\documentclass[12pt]{extarticle}
\usepackage{amsmath,amsthm,amssymb,amscd,mathtools,color,graphicx,subcaption,hyperref,enumitem}
\usepackage{algorithm}
\usepackage{algpseudocode}
\usepackage[all]{xy}
\usepackage{tikz}
\oddsidemargin \evensidemargin
\newtheorem{theorem}{Theorem}
\numberwithin{theorem}{section}
\newtheorem{proposition}[theorem]{Proposition}
\newtheorem{lemma}[theorem]{Lemma}
\newtheorem{corollary}[theorem]{Corollary}
\newtheorem{definition}[theorem]{Definition}
\newtheorem{remark}[theorem]{Remark}

\newtheorem{problem}[theorem]{Problem}
\newtheorem{example}[theorem]{Example}

\begin{document}
\title{Commutative Algebra of Generalised Frobenius Numbers}
\author{Madhusudan Manjunath \footnote{Part of this work was carried out while MM was at Queen Mary University of London, the Mathematisches Forschungsinstitut Oberwolfach and IHES, Bures-sur-Yvette.  He was funded by the EPSRC at QMUL and by a Leibniz Fellowship at the MFO}  and Ben Smith}
\maketitle

\begin{abstract}
We study commutative algebra arising from generalised Frobenius numbers. The $k$-th (generalised) Frobenius number of relatively prime natural numbers $(a_1,\dots,a_n)$  is the largest natural number that cannot be written as a non-negative integral combination of $(a_1,\dots,a_n)$ in $k$ distinct ways.  Suppose that $L$ is the lattice of integer points of $(a_1,\dots,a_n)^{\perp}$. Taking cue from the concept of lattice modules due to Bayer and Sturmfels, we define generalised lattice modules $M_L^{(k)}$ whose Castelnuovo--Mumford regularity captures the $k$-th Frobenius number of $(a_1,\dots,a_n)$. We study the sequence $\{M_L^{(k)}\}_{k=1}^{\infty}$ of generalised lattice modules providing an explicit characterisation of their minimal generators. We show that there are only finitely many isomorphism classes of generalised lattice modules. As a consequence of our commutative algebraic approach, we show that the sequence of generalised Frobenius numbers forms a finite difference progression i.e. a sequence whose set of successive differences is finite. We also construct an algorithm to compute the $k$-th Frobenius number.
\end{abstract}

\section{Introduction}

The Frobenius number $F(a_1,\dots,a_n)$ of a collection $(a_1,\dots,a_n)$ of natural numbers such that ${\rm gcd}(a_1,\dots,a_n)=1$ is the largest natural number that cannot be expressed as a non-negative integral linear combination of $a_1,\dots,a_n$. Note that the condition ${\rm gcd}(a_1,\dots,a_n)=1$ ensures that a sufficiently large integer can be written as an non-negative integral combination of $a_1,\dots,a_n$. Note that throughout the paper, we use the convention that $\mathbb{N} = \{1,2,3,\dots\}$. The Frobenius number has been studied extensively from several viewpoints including discrete geometry \cite{Kan92}, analytic number theory \cite{BecDiaRob02} and commutative algebra \cite{RahSab06}.
 
The Frobenius number can be rephrased in the language of lattices as follows \cite{ScaSha93}. We start by letting $L(a_1,\dots,a_n)$ be a sublattice of the dual lattice $(\mathbb{Z}^n)^\star$ of points that evaluate to zero at $(a_1,\dots,a_n) \in \mathbb{Z}^n$. The Frobenius number is precisely the largest integer $r$ such that there exists a point ${\bf p} \in (\mathbb{Z}^n)^\star$ that evaluates to $r$ at  $(a_1,\dots,a_n)$ and ${\bf p}$ does not dominate any point in $L(a_1,\dots,a_n)$. Here the domination is according to the partial order induced by the standard basis on $(\mathbb{Z}^n)^\star$.

This leads to a commutative algebraic interpretation of the Frobenius number that we now recall. Let $\mathbb{K}$ be an arbitrary field and let $S=\mathbb{K}[x_1,\dots,x_n]$ be the polynomial ring in $n$ variables with coefficients in $\mathbb{K}$.  Let $I_{L(a_1,\dots,a_n)}$ or simply $I_L$ be the lattice ideal associated to $L$. Recall that for a sublattice $L$ of $\mathbb{Z}^n$, the lattice ideal $I_L$ is the ideal generated by all binomials ${\bf x^u}-{\bf x^v}$  such that ${\bf u}-{\bf v} \in L$ and ${\bf u},~{\bf v} \in \mathbb{Z}_{\geq 0}^n$. Note that $L(a_1,\dots,a_n)$ is, by construction, a sublattice of $(\mathbb{Z}^n)^{\star}$. We use the standard isomorphism between $(\mathbb{Z}^n)^{\star}$ and $\mathbb{Z}^n$ to regard it as a sublattice of $\mathbb{Z}^n$ and associate a lattice ideal to it. Observe that the $\mathbb{Z}^n$-grading on $S/I_L$ also yields a $\mathbb{Z}$-grading, corresponding to the evaluation $\sigma \rightarrow \sigma((a_1,\dots,a_n))$ for $\sigma \in (\mathbb{Z}^n)^*$. We refer to this grading as the $(a_1,\dots,a_n)$-weighted grading or simply the weighted grading.

\begin{theorem}\cite{RahSab06} \label{frobcomalg_theo}
The Frobenius number $F(a_1,\dots,a_n)$ is the given by the formula:
\[
{\rm reg}(S/I_L)+n-1-\sum_{i=1}^n a_i
\]
where ${\rm reg}(S/I_L)$ is the Castelnuovo--Mumford regularity of $S/I_L$ with respect to its $(a_1,\dots,a_n)$-weighted grading.  In other words, the Frobenius number is the maximum weighted degree of the highest Betti number of $I_L$ as an $S$-module subtracted by $\sum_i a_i$.
\end{theorem}

\begin{remark}
\rm 
For a module $M$ with weighted grading, the invariant ${\rm reg}(S/I_L)+n-1-\sum_{i=1}^n a_i$ is also called the $a$-invariant of the module,  \cite{RahSab06}. \qed
\end{remark}
\begin{example}
\rm
Consider the lattice $L = (3,5,8)^{\perp} \cap \mathbb{Z}^3$. We calculate its corresponding lattice ideal $I_L = \left \langle x_3 - x_1x_2, x_2^3 - x_1^5 \right \rangle$. The Betti table corresponding to the minimal free resolution of $S/I_L$ has 22 rows and 3 columns, hence $\text{reg}(S/I_L) = 21$ and $F(3,5,8) = 21 + 2 - 16 = 7$. \qed
\end{example}

Throughout this paper by the Castelnuovo--Mumford regularity of a graded module $M$, we mean the maximum row index in its graded Betti table minus one. If $c_{i,j}$ is the twist corresponding to the Betti number $\beta_{i,j}$, the regularity is given by ${\rm max}_{i,j} \{c_{i,j}-i\}$. We refer to Eisenbud \cite[Chapter 4]{Eis02} for more information on this topic.

Theorem \ref{frobcomalg_theo} motivates studying ``explicit'' free resolutions of $I_L$ as an $S$-module.  By an explicit free resolution, we mean a cell complex on $L$ whose relabeling gives a free resolution. For instance, the hull complex \cite{BayStu98} gives an explicit (non-minimal, in general) free resolution. We refer to the first section of Miller and Sturmfels \cite{MilStu05} for more information.


\subsection{Generalised Frobenius Numbers}

Recently, the following generalisation of the Frobenius number called the $k$-th Frobenius number has been proposed \cite{BecRob03}. For a natural number $k$, the $k$-th Frobenius number $F_k(a_1,\dots,a_n)$ of a collection $(a_1,\dots,a_n)$ of natural numbers such that ${\rm gcd}(a_1,\dots,a_n)=1$ is the largest natural number that cannot be written as $k$ distinct non-negative integral linear combinations of $a_1,\dots,a_n$.  Hence, the first Frobenius number  $F_1(a_1,\dots,a_n)$ is the Frobenius number of $(a_1,\dots,a_n)$. The finiteness of $F_k(a_1,\dots,a_n)$ for all natural numbers $k$ follows by an argument similar to the one for $F_1(a_1,\dots,a_n)$.  In the language of lattices, the $k$-th Frobenius number is the largest integer $r$ such that there exists a point ${\bf p} \in (\mathbb{Z}^n)^\star$ that evaluates to $r$ at  $(a_1,\dots,a_n)$ and ${\bf p}$ does not dominate $k$ distinct points in $L(a_1,\dots,a_n)$. As in the case $k=1$, the domination is according to the partial order induced by the standard basis on $(\mathbb{Z}^n)^\star$. 

This interpretation allows a generalisation to any finite index sublattice $H$ of $L(a_1,\dots,a_n)$. The $k$-th Frobenius number of $H$ is the  largest integer $r$ such that there exists a point ${\bf p} \in (\mathbb{Z}^n)^\star$ that evaluates to $r$ at  $(a_1,\dots,a_n)$ and ${\bf p}$ does not dominate $k$ distinct points in $H$. The finite index assumption is necessary for the $k$-th Frobenius number to be finite. All our results hold in this level of generality.

Our goal in this paper is to develop commutative algebra arising from the $k$-th Frobenius number.   A guiding problem for us is the classification of sequences of generalised Frobenius numbers:

\begin{problem} \rm{\bf{(Classification of Frobenius Number Sequences)} } \label{classfrob_prob}
Given a sequence of natural numbers $\{c_n\}_{n=1}^{\infty}$, does there exist a vector $(a_1,\dots,a_n) \in \mathbb{N}^n$ and a finite index sublattice $H$ of $L(a_1,\dots,a_n)$ whose sequence of generalised Frobenius numbers is equal to $\{c_n\}_{n=1}^{\infty}$?
\end{problem}

To the best of our knowledge, this problem is wide open. For instance, previous to this it was not known whether a geometric progression with common ratio strictly greater than one can occur as a sequence of Frobenius numbers. As a corollary to our results, we show that the answer to this question is ``no''.

We start by recalling another commutative algebraic interpretation of the Frobenius number $F_1(a_1,\dots,a_n)$ following Bayer and Sturmfels \cite{BayStu98},~\cite{MilStu05}. The key concepts here are the group algebra $S[L]$ and the lattice module $M_L$ associated to $L$. 

The group algebra $S[L]$ is the  $\mathbb{K}$-algebra generated by Laurent monomials ${\bf x^u \cdot z^v}$ such that ${\bf u}=(u_1,\dots,u_n) \in \mathbb{Z}_{\geq 0}^n,~{\bf v}=(v_1,\dots,v_n) \in L$ where ${\bf x^u}$ and ${\bf z^v}$ are Laurent monomials $x_1^{u_1}\dots x_n^{u_n}$ and $z_1^{v_1}\dots z_n^{v_n}$. The lattice module $M_L$ is the $S$-module generated by Laurent monomials ${\bf x^w}$ over all ${\bf w} \in L$. Note that as an $S$-module $M_L$ is not finitely generated. However we can realise $M_L$ as a cyclic $S[L]$-module as follows: 
\begin{center}
$M_L \cong S[L]/\langle {{\bf x^u}-{\bf x^v} \cdot {\bf z^{u-v}}|~{\bf u},~{\bf v} \in  \mathbb{Z}_{\geq 0}^n,~{\bf u}-{\bf v} \in L} \rangle$
\end{center}
with the $S[L]$-action given by ${\bf x^u}{\bf z^v} \cdot{\bf  x^w}={\bf x^{w+u+v}}$ where ${ \bf x^u}{\bf z^v} \in S[L]$ and ${\bf x^w} \in M_L$. To see this isomorphism, consider the morphism from $S[L]$ to $M_L$ that takes ${\bf x^u \cdot z^v}$ to ${\bf x^{u+v}}$ and use the first isomorphism theorem.


The group algebra $S[L]$ is naturally $\mathbb{Z}^n$-graded where the graded piece indexed by ${\bf w} \in \mathbb{Z}^n$ is the $\mathbb{K}$-vector space spanned by  $\{{\bf x^u}\cdot {\bf z^v}\}$ such that ${{\bf u}+{\bf v}={\bf w},~{\bf v}\in L}$. The lattice module $M_L$ is also naturally $\mathbb{Z}^n$-graded, as it is generated by Laurent monomials. For the lattice $L(a_1,\dots,a_n) \subset (\mathbb{Z}^n)^{\star}$, we again regard it as a sublattice of $\mathbb{Z}^n$ via the standard isomorphism between $(\mathbb{Z}^n)^{\star}$ and $\mathbb{Z}^n$ to associate the group algebra and the lattice module to it. Observe that $M_L$ also carries the $(a_1,\dots,a_n)$-weighted grading, as does any lattice module whose corresponding lattice is a finite index sublattice of $L(a_1,\dots,a_n)$.

Note that $S$ is a $S[L]$-module via the isomorphism $S \cong S[L]/\langle {\bf z^v}-1_{\mathbb{K}}|~{\bf v} \in L\rangle$. Bayer and Sturmfels show that there is a categorical equivalence between $\mathbb{Z}^n$-graded $S[L]$-modules to $\mathbb{Z}^n/L$-graded  $S$-modules. The functor $\pi$ realising this equivalence tensors an $S[L]$-module with $S$ (here $S$ is seen as an $S[L]$-module). The functor $\pi$ takes $M_L$ to $S/I_L$. Hence, a $\mathbb{Z}^n/L$-minimal free resolution of $S/I_L$ as an $S$-module can be obtained by applying the functor $\pi$ to a $\mathbb{Z}^n$-graded minimal free resolution of $M_L$ as an $S[L]$-module, we refer to Miller and Sturmfels \cite{MilStu05} for an example.  Hence, we have the following interpretation of the Frobenius number in terms of the lattice module $M_L$.

\begin{theorem}\cite{BayStu98}, \cite{RahSab06}\label{froblattmod_theo}
The Frobenius number $F(a_1,\dots,a_n)$ is 
\[
{\rm reg}(M_L)+n-1-\sum_{i=1}^{n}a_i
\]
where ${\rm reg}(M_L)$ is the Castelnuovo--Mumford regularity of $M_L$ with respect to its $(a_1,\dots,a_n)$-weighted grading.
\end{theorem}

The module $M_L$ behaves similar to a monomial ideal. This categorical equivalence can be used to transfer homological constructions from $M_L$ to $S/I_L$. By applying the functor $\pi$ to $M_L$ and noting that the $\mathbb{Z}^n/L$-grading coincides with the $(a_1,\dots,a_n)$-weighted grading on $L$, we obtain Theorem \ref{frobcomalg_theo}. We start by generalising Theorem \ref{froblattmod_theo} to $k$-th Frobenius numbers. We generalise the lattice module $M_L$ to the $k$-th lattice module $M^{(k)}_L$ as follows.


\begin{definition}
The $k$-th lattice module $M_L^{(k)}$ is the $S$-module generated by Laurent monomials ${\bf x^w}$ such that ${\bf w}$ dominates at least $k$ lattice points.
\end{definition}

\begin{figure}
\centering
\includegraphics[width=0.6\columnwidth]{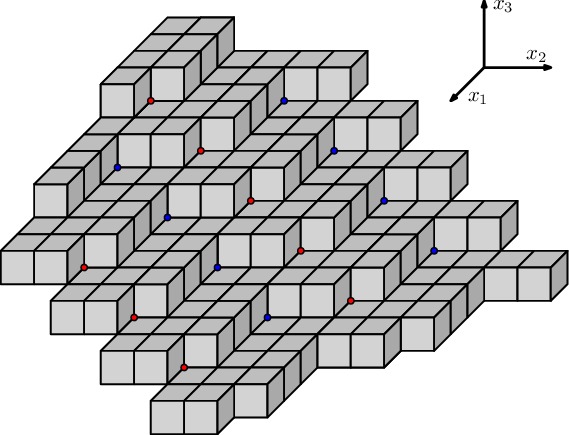}%
\caption{The staircase diagram corresponding to the 3rd lattice module $M_{L(3,4,11)}^{(3)}$.  The polynomial ring $S = \mathbb{K}[x_1,x_2,x_3]$ has the $(3,4,11)$-weighted grading. The blue dots correspond to minimal generators of degree 15, red dots correspond to minimal generators of degree 20.}
\label{fig:MonomialStaircase}%
\end{figure}

Figure \ref{fig:MonomialStaircase} shows the staircase diagram corresponding to the 3rd lattice module of the lattice $L(3,4,11)$. By construction, the first lattice module $M_L^{(1)}$ is the lattice module $M_L$. The module $M^{(k)}_L$ is also not finitely generated as an $S$-module, however it can be viewed as a finitely generated $S[L]$-module (see Proposition \ref{fingen_prop} for more details) with the $S[L]$-action given by ${\bf x^u}{\bf z^v} \cdot{\bf  x^w}={\bf x^{w+u+v}}$ where ${ \bf x^u}{\bf z^v} \in S[L]$ and ${\bf x^w} \in M^{(k)}_L$. The generalised lattice module $M^{(k)}_L$  also carries a $\mathbb{Z}^n$-grading since it is generated by Laurent monomials.  However, in general the module $M^{(k)}_L$ is not a cyclic module for natural numbers $k >1$.  We have the following commutative algebraic characterisation of generalised Frobenius numbers in terms of the generalised lattice modules.

\begin{proposition}\label{kfroblattmod_prop}
The $k$-th Frobenius number of $(a_1,\dots,a_n)$ is given by the formula:

\[
{\rm reg}(M^{(k)}_L)+n-1-\sum_{i=1}^{n}a_i
\]
where $L:=L(a_1,\dots,a_n)$ and ${\rm reg}(M_L^{(k)})$ is the Castelnuovo--Mumford regularity of the $S[L]$-module $M^{(k)}_L$ with respect to its $(a_1,\dots,a_n)$-weighted grading.
\end{proposition}

Proposition \ref{kfroblattmod_prop} follows from two observations. Let $c_{i,j}$ be the twist of the free module corresponding to the graded Betti number $\beta_{i,j}$ of $\pi(M_L^{(k)})$. A computation similar to the proof of \cite[Theorem 3.1]{RahSab06}, comparing expressions for the Hilbert series of $\pi(M_L^{(k)})$ gives the following expression for the $k$-th Frobenius number:
\begin{align*}
F_k = {\rm max}_{i,j} c_{i,j} - \sum_i a_i
\end{align*}

The second observation is that $\pi(M_L^{(k)})$ is a Cohen--Macaulay module with both Krull dimension and depth equal to one. This implies that the regularity and $\text{max}_{i,j} c_{i,j}$ are attained at the highest homological degree, in this case $n-1$. Proposition \ref{kfroblattmod_prop} follows as an immediate consequence.

While Proposition \ref{kfroblattmod_prop} provides a simple description of the generalised Frobenius number in terms of $M_L^{(k)}$, it raises a number of questions. For instance, given a natural number $k$, how can we use Proposition \ref{kfroblattmod_prop}   to determine the $k$-th Frobenius number.  This requires a more explicit knowledge of $M_L^{(k)}$. The generalised lattice modules are naturally related by the filtration:

\begin{center}
$M_L^{(1)} \supseteq M_L^{(2)} \supseteq M_L^{(3)} \dots$
\end{center}

How does this filtrartion control their Castelnuovo--Mumford regularity?  The connection between the generalised Frobenius numbers can be better understood by studying the connection between the generalised lattice modules. With this in mind, we delve into a detailed study of the generalised lattice modules: their minimal generating sets, their Hilbert series and their syzygies. We now summarise our results.

We associate a graph $G_L$  on $L$ as follows. Fix a binomial minimal generating set of $I_L$. The graph $G_L$ is defined as follows: there is an edge between points ${\bf w_1}$ and ${\bf w_2}$ in $L$ if there exists a binomial minimal generator ${\bf x^{u}-x^{v}}$ such that the difference of its exponents is equal to ${\bf w_1}-{\bf w_2}$ i.e. ${\bf u}-{\bf v}={\bf w_1}-{\bf w_2}$. By construction, $G_L$ has an $L$-action on its edges since if $({\bf w_1},{\bf w_2})$ is an edge then $({\bf w_1}+{\bf y}, {\bf w_2}+{\bf y})$ for any ${\bf y} \in L$.

Let $d_{G_L}$ be the metric on $L$ induced by the graph $G_L$.  For a point ${\bf w} \in L$, let $N^{(k)}({\bf w})$ be the set of all points in $L$ in the ball of radius $k$ centered at ${\bf w}$ in the metric $d_{G_L}$.


\begin{theorem}\label{gensklattmod_thm}{\rm (\bf{Neighbourhood Theorem})}
For any non-negative integer $k$, any minimal generator of $M_L^{(k+1)}$ as an $S[L]$-module is the least common multiple of Laurent monomials corresponding to $(k+1)$ lattice points each of which is a point in $N^{(k)}({\bf 0})$ where ${\bf 0}=(0,\dots,0)$.
\end{theorem}

We prove Theorem \ref{gensklattmod_thm} by an inductive characterisation of the generalised lattice modules. This characterisation of $M_L^{(k+1)}$ is in terms of syzygies of $M_L^{(k)}$ that we believe is of independent interest. 
We briefly describe this characterisation in the following. Fix a natural number $k$, a minimal generator ${\bf x^w}$ of $M_L^{(k)}$ as an $S$-module is called exceptional if ${\bf w}$ dominates strictly larger than $k$ points in $L$.
We describe $M_L^{(k+1)}$ in terms of the exceptional generators of $M_L^{(k)}$ and the first syzygies of a ``modification" of $M_L^{(k)}$ that we now describe. Let $M_{L,{\rm mod}}^{(k)}$ be the $S$-module generated by every element of $M_L^{(k)}$ and the element $1_{\mathbb{K}}$ (the multiplicative identity of $\mathbb{K}$). Formally, \begin{center} $M_{L,{\rm mod}}^{(k)}=\langle 1_{\mathbb{K}}, m|~ m \in M_L^{(k)} \rangle_S$ \end{center}

Note that  for $k>1$, $M_{L,{\rm mod}}^{(k)}$ is naturally an $S$-module but not an $S[L]$-module i.e. $M_{L,{\rm mod}}^{(k)}$ does not inherit the natural $L$-action. 
 
 By construction, we have the following characterisation of minimal generators of $M_{L,{\rm mod}}^{(k)}$:

\begin{proposition}\label{mlmod_prop}
The \textup{(}Laurent\textup{)} monomial minimal generating set of $M_{L,{\rm mod}}^{(k)}$ consists of precisely  $1_{\mathbb{K}}$ and every \textup{(}Laurent\textup{)} monomial minimal generator of $M_L^{(k)}$ that is not divisible by $1_{\mathbb{K}}$ \textup{(}in other words, whose exponent does not dominate the origin\textup{)}.  
\end{proposition}

For each minimal generator $g_1$ of $M_{L,{\rm mod}}^{(k)}$, let ${\rm Syz}^1_{g_1}(M_{L,{\rm mod}}^{(k)})$ be the $\mathbb{K}$-vector space ${\rm Syz}^1_{g_1}(M_{L,{\rm mod}}^{(k)})$ generated by syzygies of the form:

\begin{center}

 $m\cdot(0,\dots,0,\underbrace{{\rm lcm}(g_1,g_2)/g_1}_{g_1},0,\dots,0,\underbrace{-{\rm lcm}(g_1,g_2)/g_2}_{g_2},0,\dots,0)$ 
 
\end{center}
 where $g_2 \neq g_1$  is a minimal generator of  $M_{L,{\rm mod}}^{(k)}$ and $m$ is a monomial in $S$.  Note that multiplication by $m$ is the standard multiplication on $S$.  Consider the direct sum $\oplus_{g}{\rm Syz}^1_{g}(M_{L,{\rm mod}}^{(k)})$ where $g$ varies over all minimal generators of $M_{L,{\rm mod}}^{(k)}$.

We define a map $\phi^{(k)}_S$ from $\oplus_{g}{\rm Syz}^1_{g}(M_{L,{\rm mod}}^{(k)})$ to $M_L^{(k+1)}$. We first define the map $\phi^{(k)}_S$ from the canonical basis of each piece ${\rm Syz}^1_{g}(M_{L,{\rm mod}}^{(k)})$ to $M_L^{(k+1)}$ as follows: 

\begin{center}$\phi^{(k)}_S((s_1,s_2))={\bf x}^{{\rm deg}_{\mathbb{Z}^n}((s_1,s_2))}$ where ${\rm deg}_{\mathbb{Z}^n}(.)$ is the multidegree.\end{center}

We extend this map $\mathbb{K}$-linearly to define $\phi^{(k)}_S$. Note that the image of $\phi^{(k)}_S$ is an element of $M_L^{(k+1)}$. This is because $(s_1,s_2)$ is of the form $({\rm lcm}(g_1,g_2)/g_2,-{\rm lcm}(g_1,g_2)/g_1)$ for two distinct minimal generators of $M_{L,{\rm mod}}^{(k)}$. By construction,  $\phi^{(k)}_S((s_1,s_2))={\rm lcm}(g_1,g_2)$. By Proposition \ref{mlmod_prop2}, we have the following two cases:  either both $g_1$ and $g_2$ are minimal generators of $M_L^{(k)}$ or one of them $g_1=1_{\mathbb{K}}$, say and $g_2$ is a minimal generator of $M_L^{(k)}$ that is not divisible by $1_{\mathbb{K}}$. In both cases, the support of ${\rm lcm}(g_1,g_2)$ contains at least $(k+1)$ points in $L$ (by support of Laurent monomial, we mean the set of points in $L$ that its exponent dominates). It contains (potentially among others) the  unions of the supports of $g_1$ and $g_2$. Hence, the image of $\phi^{(k)}_S$ is in $M_L^{(k+1)}$. We show the following converse to this.

\begin{theorem}\label{inductchar_theo}
Up to the action of $L$,  every minimal generator of $M_L^{(k+1)}$ is either in the image of $\phi^{(k)}_S$ or is an exceptional generator of $M_L^{(k)}$.
\end{theorem}

\begin{example} \label{phione_ex}
\rm
Consider the lattice $(3,5,8)^{\perp} \cap \mathbb{Z}^3$ with corresponding lattice ideal $I_L = \left\langle x_3 - x_1x_2, x_2^3 - x_1^5 \right\rangle$. As $1_{\mathbb{K}}$ is a minimal generator of $M_L^{(1)}$, the lattice module is not altered under the modification construction. The minimal first syzygies of $M_L^{(1)}$, up to the action of $L$, are of the form $(\text{lcm}({\bf x^u}, 1_{\mathbb{K}})/1_{\mathbb{K}}, -\text{lcm}({\bf x^u}, 1_{\mathbb{K}})/{\bf x^u})$ where ${\bf u} \in N^{(1)}({\bf 0})$. The map $\phi_{S}^{(1)}$ sends the minimal first syzygies to $\{ x_3, x_1x_2, x_2^3, x_1^5 \}$, precisely the monomials in each minimal binomial of $I_L$. This gives us an explicit description of $M_L^{(2)} = \left\langle x_3, x_2^3 \right\rangle$ as an $S[L]$-module and a minimal generating set.
\qed
\end{example}

Theorem \ref{inductchar_theo} characterises the minimal generators of $M_L^{(k)}$, a natural next question is about the syzygies of $M_L^{(k)}$ as an $S[L]$-module. Is there a similar inductive characterisation of the syzygies of $M_L^{(k)}$?
What are the possible Betti tables of $M_L^{(k)}$? Both these questions are wide open in general. As a first result in this direction, we show the following finiteness result. Recall that for a $\mathbb{Z}^n$-graded module $M$ (either an $S$-module or an $S[L]$-module) and for any ${\bf b} \in \mathbb{Z}^n$, we have the twist $M({\bf b})$ of $M$ defined by $(M_{\bf b})_{\bf c}=M_{{\bf b+c}}$ for every ${\bf c} \in \mathbb{Z}^n$.

\begin{theorem}\label{finiteness_theo}  Let $L$ be a finite index sublattice of $(a_1,\dots,a_n)^{\perp} \cap \mathbb{Z}^n$.  For each $k \in \mathbb{N}$, let ${\bf x^{u_k}}$ be any element of $M_L^{(k)}$ of the smallest $(a_1,\dots,a_n)$-weighted degree.   There are finitely many classes among the generalised lattice modules $\{{M_L}^{(k)}{({\bf u_k})}\}_{k \in \mathbb{N}}$ up to isomorphism of both $\mathbb{Z}^n$-graded $S[L]$-modules and $\mathbb{Z}^n$-graded $S$-modules. Hence, there are only finitely many distinct Betti tables for the generalised lattice modules of $L$.
\end{theorem}

The key object in the proof of Theorem \ref{finiteness_theo} is a poset that we refer to as the structure poset associated to $L$.  The elements of the structure poset of $L$ are elements in $\mathbb{Z}^n/L$ of  $(a_1,\dots,a_n)$-weighted degree in the range $[0,F_1]$ where $F_1$ is the first Frobenius number of $L$. Note that there are precisely ${\rm ind}(L) \cdot (F_1+1)$ elements in the structure poset, where ${\rm ind}(L)$ is the index of $L$ in $(a_1,\dots,a_n)^{\perp} \cap \mathbb{Z}^n$. The partial order in this poset is defined as follows: for elements $[{\bf a}],~[{\bf b}]$ in the structure poset we say that $[{\bf a}] \geq [{\bf b}]$ if for every representative ${\bf a} \in \mathbb{Z}^n$ of $[{\bf a}]$ there exists a representative ${\bf b}$ of $[{\bf b}]$ such that ${\bf a} \geq {\bf b}$ under the standard partial order on $\mathbb{Z}^n$.

Let $m_k$ be the minimum weighted degree of any element of $M_L^{(k)}$. The key observation in the proof of Theorem \ref{finiteness_theo} is that $M_L^{(k)}$ is completely determined (up to isomorphism of $\mathbb{Z}^n$-graded $S[L]$-modules) by ``filling'' the structure poset of $L$. More precisely, to determine $M_L^{(k)}$ we need to know the elements in $\mathbb{Z}^n/L$ of weighted degree $[m_k,m_k+F_1]$ that dominate at least $k$ points in $L$. 
These elements determine a subposet of the structure poset of $L$ that we refer to as the {\it structure poset of $M_L^{(k)}$}.

The structure poset of $M_L^{(k)}$ determines it up to isomorphism of $\mathbb{Z}^n$-graded $S[L]$-modules.  Since the structure poset of $L$ is finite it has only finitely many subposets. Hence, there are only finitely many $\mathbb{Z}^n$-graded isomorphism classes of generalised lattice modules. A classification of the structure poset of $L$ and subposets of the structure poset of $L$ that can occur as structure posets of generalised lattice modules is wide open. In Section \ref{finiteness_sect}, we provide a detailed example illustrating this phenomenon.  As a corollary to Theorem \ref{finiteness_theo} we obtain the following:

\begin{corollary}\label{finitenessfrob_cor}
There exists a finite set of integers $\{b_1,\dots,b_t\} \subset \mathbb{Z}_{\geq 0} \cup \{-1\}$ such that for every $k$ there exists a natural number $j$ such that the $k$-th Frobenius number can be written as:  \begin{center} $F_k=m_k+b_j$ \end{center} where $m_k$ is the minimum $(a_1,\dots,a_n)$-weighted degree of any element in $M_L^{(k)}$. This finite set $\{b_1,\dots,b_t\}$ is precisely the set of  integers that can be realised  as ${\rm reg}(M_L^{(k)}({\bf u_k}))+n-1-\sum_{i=1}^n a_i$.
\end{corollary}

For $k=1$, note that $m_k=0$ and $b_j = {\rm reg}(M_L^{(1)}) + n - 1 - \sum a_i$. Suppose that $L$ is a finite sublattice of $(a_1,a_2)^{\perp} \cap \mathbb{Z}^2$ where $(a_1,a_2)$ are relatively prime numbers.  
The set $\{b_1,\dots,b_t\}$ consists of one element and the generalised lattice module $M_L^{(k)}$ is generated by one element. When $L = (a_1,a_2)^{\perp} \cap \mathbb{Z}^2$, the $k$-th lattice module $M_L^{(k)}$ is generated by $x_1^{a_2}x_2^{(k-2)a_1}$. Hence, $m_k=(k-1)a_1a_2$ and the set $\{b_1,\dots,b_t\}$ consists of only one element $F_1=a_1a_2-a_1-a_2$. The $k$-th Frobenius number $F_k$ will be $m_k+F_1=ka_1a_2-a_1-a_2$, exactly as in \cite{BecRob03}. Obtaining a formula along the lines of Corollary \ref{finitenessfrob_cor} was suggested as an open problem in \cite{BecRob03}. Finally, note that if the structure poset of $M_L^{(k)}$ is equal to the structure poset of $L$ then 
$F_k=m_k-1$ and in this case $b_j=-1$.

As an application of Theorem \ref{finiteness_theo}, we show that the sequence of Frobenius numbers $(F_k)_{k=1}^{\infty}$ is a finite difference progression. A sequence $(c_k)_{k=1}^{\infty}$ is called a finite difference progression if there exists a finite set of differences such that for every $k \in \mathbb{N}$ the difference $c_{k+1}-c_{k}$ is contained in this set. This provides a partial answer to Problem \ref{classfrob_prob}.

\begin{theorem} \label{frogenarith_theo}
For any finite index sublattice $L$ of $(a_1,\dots,a_n)^{\perp} \cap \mathbb{Z}^n$, the sequence of Frobenius numbers $(F_k)_{k=1}^{\infty}$ is a finite difference progression.
 \end{theorem} 
 
 We prove Theorem \ref{frogenarith_theo} using Corollary \ref{finitenessfrob_cor} along with the fact that the sequence $(m_k)_{k=1}^{\infty}$ is a finite difference progression.  As a corollary, a geometric progression with common ratio strictly greater than one cannot occur as a sequence of Frobenius numbers.

 As an another application of our results, we use the neighbourhood theorem (Theorem \ref{gensklattmod_thm}) to construct an algorithm that takes the lattice in terms of a basis and a natural number $k$ as input, computes the $k$-th lattice module and the $k$-th Frobenius number.












\subsection{Related Work}

There is a vast literature on the Frobenius number, we refer to Alfons\'{i}n's book \cite{Alf06} for more information. Work on the generalised Frobenius numbers has so far primarily used analytic methods and methods from polyhedral geometry.  The work of Beck and Robins \cite{BecRob03} uses analytic methods to derive an explicit formula for the coefficients of the Hilbert series of $\mathbb{K}[x,y]$ with the $(a_1,a_2)$-weighted grading.  Aliev, Fukshanksy and Henk \cite{AliFukHen12} give bounds generalising a theorem of Kannan for the first Frobenius number. They relate the $k$-th Frobenius number to the $k$-covering radius of a simplex with respect to the lattice $(a_1,\dots,a_n)^{\perp} \cap \mathbb{Z}^n$,  giving bounds on the generalised Frobenius number as a corollary.

A recent work of  Aliev, De Loera and Louveaux \cite{AliLoeLou16} considered the semigroup
\[
{\rm Sg}_{\geq k}((a_1,\dots,a_n)) = \{b : \exists \; {\bf x_1},\dots, {\bf x_k} \in \mathbb{Z}_{\geq 0}^n \text{ such that } (a_1,\dots,a_n) \cdot {\bf x_i} = b\}
\]
where ${\bf x_1},\dots, {\bf x_k}$ are distinct.
In this framework, the $k$-th Frobenius number is the largest non-negative integer $b \notin {\rm Sg}_{\geq k}((a_1,\dots,a_n))$. They study this semigroup by considering a monomial ideal $I^{(k)}((a_1,\dots,a_n))$ such that the set of $(a_1,\dots,a_n)$-weighted degrees of its elements is equal to ${\rm Sg}_{\geq k}((a_1,\dots,a_n))$ \cite[Theorem 1]{AliLoeLou16}.  They use the Gordan--Dickson Lemma to deduce the finite generation of $I^{(k)}((a_1,\dots,a_n))$ and hence, ${\rm Sg}_{\geq k}((a_1,\dots,a_n))$. In fact,  \cite{AliLoeLou16} study a more general version where $(a_1,\dots,a_n)$ is replaced by any $d \times n$ matrix with integer entries.

The monomial ideal $I^{(k)}((a_1,\dots,a_n))$ is, in fact, the intersection of $M_L^{(k)}$ with the polynomial ring $S$ (both are submodules of $S[L])$. We note that this ideal $I^{(k)}((a_1,\dots,a_n))$ does not carry an $L$-action and this seems to make it less amenable to study compared to $M_L^{(k)}$. 

\section{Generalised Lattice Modules}
In this section, we discuss generalised lattice modules in detail including the neighbourhood theorem (Theorem \ref{gensklattmod_thm}) and the inductive characterisation of $M_L^{(k)}$ (Theorem \ref{inductchar_theo}) in the introduction. We start by recalling the definition of generalised lattice modules. Fix a non-zero vector $(a_1,\dots,a_n) \in \mathbb{N}^n$. Let $L$ be a finite index sublattice of the lattice of integer points in $(a_1,\dots,a_n)^{\perp} \cap \mathbb{Z}^n$ and $S=\mathbb{K}[x_1,\dots,x_n]$ be the polynomial ring in $n$-variables.

\begin{definition}
The \emph{$k$-th lattice module} $M_L^{(k)}$ is the $S$-module generated by Laurent monomials ${\bf x^w}$ where ${\bf w}$ is an element in $\mathbb{Z}^n$ that dominates at least $k$ points in $L$.  Formally,
\begin{center}
$M_L^{(k)}=\langle {\bf x^w}|$~ ${\bf w} \in \mathbb{Z}^n$ dominates at least $k$ points in $L \rangle_S$
\end{center}
\end{definition}
By construction, $M_L^{(k)}$ is a $\mathbb{Z}^n$-graded $S$-module and is $\mathbb{Z}$-graded by $(a_1,\dots,a_n)$-weighted degree.  On the other hand,  $M_L^{(k)}$ is not a finitely generated $S$-module. However $M_L^{(k)}$ carries an $L$-action via the map  ${\bf v} \cdot {\bf x^w}={\bf x^{w+v}}$ for every ${\bf v} \in L$ and ${\bf x^w} \in M_L^{(k)}$. This action makes $M_L^{(k)}$ into a $L$-module and furthermore, into an $S[L]$-module where $S[L]$ is the group algebra of $L$. 

Recall that the group algebra $S[L]$ is defined as the $\mathbb{K}$-algebra generated by symbols ${\bf x^u}{\bf z^v}$ where ${\bf u} \in \mathbb{Z}_{\geq 0}^n$ and ${\bf v} \in L$ with multiplication given by ${\bf x^{u_1}}{\bf z^{v_1}} \cdot {\bf x^{u_2}}{\bf z^{v_2}}={\bf x^{u_1+u_2}}{\bf z^{v_1+v_2}}$. 
The action by $S[L]$ on the $k$-th lattice module $M_L^{(k)}$ is given by ${\bf x^u}{\bf z^v} \cdot {\bf x^w}={\bf x^{u+v+w}}$ where ${\bf x^u}{\bf z^v} \in S[L]$ and ${\bf x^w} \in M_L^{(k)}$.
 We refer to \cite{BayStu98}, \cite{MilStu05} for a more detailed discussion on this topic.  In the following, we show that $M_L^{(k)}$ is a finitely generated $S[L]$-module.
 
 \begin{proposition}\label{fingen_prop} For any natural number $k$, the $k$-th lattice module $M_L^{(k)}$ is a finite generated $S[L]$-module. \end{proposition}
 \begin{proof}
 By the action of $S[L]$ on $M_L^{(k)}$, it suffices to consider orbit representatives of the $L$-action on $M_L^{(k)}$ that dominate the origin. These representatives are monomials in $S$ (rather than Laurent monomials) and define a monomial ideal in the polynomial ring $S$. By the Gordan--Dickson Lemma, this monomial ideal is finitely generated and hence $M_L^{(k)}$ is finitely generated as an $S[L]$-module.
 \end{proof}
 
The  proof of Proposition \ref{fingen_prop} is based on an argument in \cite{AliLoeLou16}, it is however not constructive in the sense that it does not give bounds on the degrees of the minimal generators of $M_L^{(k)}$. The methods in Section \ref{finiteness_sect} give a constructive proof that shows that the $(a_1,\dots,a_n)$-weighted degree of any minimal generator of $M_L^{(k)}$ is in the interval $[m_k,m_k+F_1]$ where $m_k$ is the minimum $(a_1,\dots,a_n)$-weighted degree of a Laurent monomial ${\bf x^w}$ such that ${\bf w}$ dominates at least $k$ lattice points.
 
\begin{example} \rm The case $k=1$ is precisely the notion of lattice module studied by Bayer and Sturmfels \cite{BayStu98}.  The lattice module $M_L^{(1)}$ as an $S$-module is generated by Laurent monomials ${\bf x^w}$ where ${\bf w} \in L$. As an $S[L]$-module, $M_L^{(1)}$ is cyclic and  is generated by the element $1_{\mathbb{K}}$. Furthermore, \cite{BayStu98} show that $M_L^{(1)} \cong S[L]/ \langle {\bf x^{u}}-{\bf x^v}{\bf z^{u-v}}|~{\bf u},~{\bf v} \in \mathbb{Z}_{\geq 0}^n,~{\bf u-v} \in L \rangle$. \qed
\end{example}

For $k \geq 2$, the lattice modules $M_L^{(k)}$ are in general not cyclic $S[L]$-modules. For $k=2$, we give a simple description of a minimal generating set of $M_L^{(2)}$ in terms of the first syzygies of $M_L^{(1)}$  (Theorem \ref{mltwochar_theo}). For $k \geq 3$, a generalisation of this result is more involved and is the content of Theorem \ref{indchar_theo2}. One source of complication is that for $k \geq 2$, the lattice modules $M_L^{(k)}$ have exceptional generators i.e. those that dominate strictly greater than $k$ points in $L$, whereas $M_L^{(1)}$ does not have exceptional generators.  Another complication is for $k \geq 3$, we may get minimal generators of $M_L^{(k)}$ that do not arise as a syzygy between two minimal generators of $M_L^{(k-1)}$, rather as a ``syzygy between a minimal generator and a lattice point''. This motivates us to consider the syzygies of $M_{L,{\rm mod}}^{(k)}$.


\subsection{Inductive Characterisation of $M_L^{(2)}$}

We discuss the simplest generalised lattice module $M_L^{(2)}$. We start with a description of the minimal generators of $M_L^{(2)}$. Recall from the introduction that the key to this description is the morphism $\phi_S^{(k)}$ between $\oplus_g{\rm Syz}^{1}_g(M_{L,\text{mod}}^{(k)})$ and $M_L^{(k+1)}$.  We now describe the specialisation of this map for $k=1$. We first note that $M_L^{(1)} = M_{L,\text{mod}}^{(1)}$ and that each piece ${\rm Syz}^1_{\bf x^u}(M_L^{(k)})$  has a basis of the form:
\begin{center}$m \cdot (0,\dots,0,\underbrace{{\rm lcm}({\bf x^u},{\bf x^v})/{\bf x^u}}_{{\bf x^{u}}}, \dots,-\underbrace{{\rm lcm}({\bf x^u },{\bf x^v})/{\bf x^v}}_{{\bf x^v}},0,\dots,0)$, \end{center}
where ${\bf x^v} \neq {\bf x^u}$ is a Laurent monomial minimal generator in $M_L^{(1)}$ as an $S$-module, ${\rm lcm}(.,.)$ is the least common multiple and $m$ is a monomial in $S$.  Note that multiplication by $m$ is the standard multiplication on $S$. 

We define a map $\phi_S^{(1)}$ on this basis of ${\rm Syz}_{\bf x^u}^{1}(M_L^{(1)})$ and extend it $\mathbb{K}$-linearly. The map $\phi_S^{(1)}:{\rm Syz}^{1}_{\bf x^u}(M_L^{(1)}) \rightarrow M_L^{(2)}$ takes the element  $s=({\rm lcm}({\bf x^u},{\bf x^v})/{\bf x^u}, -{\rm lcm}({\bf x^u},{\bf x^v})/{\bf x^v})$ to ${\bf x}^{{\rm deg}_{\mathbb{Z}^n}(s)}$ where ${\rm deg}_{\mathbb{Z}^n}(s)$ is the $\mathbb{Z}^n$-graded degree of $s$. In fact, ${\rm deg}_{\mathbb{Z}^n}(s)={\rm max}({\bf u},{\bf v})$ where max is the coordinate-wise maximum. Furthermore, ${\bf x}^{{\rm deg}_{\mathbb{Z}^n}(s)} \in M_L^{(2)}$ since the point ${\rm max}({\bf u},{\bf v})$ dominates at least two lattice points, namely ${\bf u}$ and ${\bf v}$.  In the following, we note that the map $\phi_S^{(1)}$ is surjective.

\begin{proposition}\label{phi1_surj}
 The map $\phi_S^{(1)}$ is surjective.
\end{proposition}
\begin{proof}
It suffices to prove that every Laurent monomial in $M_L^{(2)}$ can be realised as the image of an element in ${\rm Syz}_{\bf x^u}^{1}(M_L^{(1)})$ for some minimal generator ${\bf x^u}$ of $M_L^{(1)}$. 
To see this, consider a Laurent monomial ${\bf x^w}$ in $M_L^{(2)}$. By the definition of $M_L^{(2)}$, the point ${\bf w}$ dominates at least two points in $L$. Consider any two points ${\bf u_1}$ and ${\bf u_2}$ in $L$ that ${\bf w}$ dominates and consider the Laurent monomial  ${\rm lcm}\left({\bf x^{u_1}}, {\bf x^{u_2}}\right)$. This is contained in $M_L^{(2)}$ and is the image of $({\rm lcm}({\bf x^{u_1}}, {\bf x^{u_2}})/{\bf x^{u_1}},-{\rm lcm}({\bf x^{u_1}}, {\bf x^{u_2}})/{\bf x^{u_2}}) \in {\rm Syz}_{\bf x^{u_1}}^1(M_L^{(1)})$ under $\phi_S^{(1)}$. Hence,  by multiplying this syzygy by the monomial ${\bf x^w}/{\rm lcm}({\bf x^{u_1}},{\bf x^{u_2}})$ we conclude that ${\bf x^w}$ is also in the image of $\phi_S^{(1)}$.
 \end{proof}
 
Proposition \ref{phi1_surj} is not directly amenable for computational purposes since $M_L^{(2)}$ is not finitely generated as an $S$-module. However,  $M_L^{(2)}$ is finitely generated as an $S[L]$-module. Note that there is a natural $L$-action on $\oplus_g{\rm Syz}^{1}_g(M_L^{(1)})$ and a surjective map between the first syzygy module of $M_L^{(1)}$ as an $S[L]$-module and the piece ${\rm Syz}^{1}_{\bf 0}(M_L^{(1)})$.  Composing this with $\phi_S^{(1)}$ gives a surjective map $\phi_{S[L]}^{(1)}$ between the first syzygy module of $M_L^{(1)}$ and $M_L^{(2)}$ as $S[L]$-modules. To explicitly describe the map $\phi_{S[L]}^{(1)}$, we first note that the first syzygy module of $M_L^{(1)}$ as an $S[L]$-module has a $\mathbb{K}$-vector space basis of the form:


\[
{\bf x^u}-{\bf x^vz^{u-v}}
\]
where ${\bf u},{\bf v} \in  \mathbb{Z}_{\geq 0}^n$ and ${\bf u-v} \in L$.  The map $\phi_{S[L]}^{(1)}$ takes ${\bf x^u}-{\bf x^vz^{u-v}}$ to 
${\bf x^u} \in M_L^{(2)}$.  As the functor $\pi$ takes $M_L^{(1)}$ to $S/I_L$ and ${\rm Syz}^1(S/I_L) = I_L$ (along with the categorical equivalence between $\mathbb{Z}^n$-graded $S[L]$-modules and $\mathbb{Z}^n/L$-graded $S$-modules), this induces a map from any binomial minimal generating set of $I_L$ to $M_L^{(2)}$, which we also refer to as $\phi_{S[L]}^{(1)}$.
We obtain the following.
\begin{theorem}\label{mltwochar_theo}
The lattice module $M_L^{(2)}$ as an $S[L]$-module is generated by the image of $\phi_{S[L]}^{(1)}$ on a binomial minimal generating set of the lattice ideal $I_L$.
\end{theorem}

As Example \ref{phione_ex} shows, the map $\phi_S^{(1)}$ is not injective. In general, take a Koszul syzygy between two minimal generators in $M_L^{(2)}$ and note that it does not lift to a syzygy between some corresponding minimal generators in  ${\rm Syz}^1(M_L^{(1)})$. This shows that $\phi_S^{(1)}$ has a non-trivial kernel and hence is not injective.

\subsection{Inductive Characterisation of $M_L^{(k)}$}

We generalise Proposition \ref{phi1_surj} to arbitrary lattice modules to obtain an induction characterisation of $M_L^{(k)}$.  Let us briefly recall the relevant objects from the introduction.

The modification $M_{L,{\rm mod}}^{(k)}$ of $M_L^{(k)}$ is the $S$-module generated by every element of $M_L^{(k)}$ and the element $1_{\mathbb{K}}$. Hence, \begin{center} $M_{L,{\rm mod}}^{(k)}=\langle 1_{\mathbb{K}}, m|~ m \in M_L^{(k)} \rangle_S$ \end{center}

By the construction of $M_{L,{\rm mod}}^{(k)}$, we have the following characterisation of its minimal generators.

\begin{proposition}\label{mlmod_prop2}
The \textup{(}Laurent\textup{)} monomial minimal generating set of $M_{L,{\rm mod}}^{(k)}$ consists of precisely  $1_{\mathbb{K}}$ and every \textup{(}Laurent\textup{)} monomial minimal generator of $M_L^{(k)}$ that is not divisible by $1_{\mathbb{K}}$ \textup{(}in other words, whose exponent does not dominate the origin\textup{)}.  
\end{proposition}


For a minimal generator $g$ of $M_{L,{\rm mod}}^{(k)}$,  the map $\phi^{(k)}_S$ from ${\rm Syz}^1_g(M_{L,{\rm mod}}^{(k)})$ to $M_L^{(k+1)}$ is defined on the canonical basis of  ${\rm Syz}^1_g(M_{L,{\rm mod}}^{(k)})$ as: 

\begin{center}$\phi^{(k)}_S((s_1,s_2))={\bf x}^{{\rm deg}_{\mathbb{Z}^n}((s_1,s_2))}$ where ${\rm deg}_{\mathbb{Z}^n}(.)$ is the multidegree of the syzygy.\end{center}
We extend the above map $\mathbb{K}$-linearly to define $\phi^{(k)}_S$. As noted in the introduction, the image of $\phi^{(k)}_S$ is contained in $M_L^{(k+1)}$.  Theorem \ref{indchar_theo2} is a converse to this.  

Suppose that ${\bf x^w} \in M_L^{(k+1)}$ is a minimal generator and let $U=\{{\bf u_1},\dots,{\bf u_{r}}\}$ be the set of points in $L$ that ${\bf w}$ dominates. For a subset $T \subset L$ of size $k$, let $\ell_T$ be the least common multiple of the Laurent monomials associated to points in $T$.

\begin{theorem}\label{indchar_theo2}
Up to the action of $L$,  any minimal generator ${\bf x^w}$ of $M_L^{(k+1)}$ is either in the image of $\phi^{(k)}_S$ or is an exceptional generator of $M_L^{(k)}$. Furthermore, we have the following classification of minimal generators of $M_L^{(k+1)}$.

\begin{enumerate}
\item If $\ell_T$ is the same for every subset $T$ of $U$ of size $k$ then, ${\bf x^w}$ is an exceptional generator of $M_L^{(k)}$.

\item If there exist subsets $T_1$ and $T_2$ of $U$ of size $k$ such that their least common multiples do not divide each other then, ${\bf x^w}$ is in image of  $\phi^{(k)}_S$ on a syzygy between two minimal generators  in $M_L^{(k)}$.

\item Otherwise, ${\bf x^w}$ is in image of  $\phi^{(k)}_S$ on a syzygy between a minimal generator in $M_L^{(k)}$ and $1_{\mathbb{K}}$.

\end{enumerate}
\end{theorem}

\begin{proof}
By definition, ${\bf w}$ dominates at least $(k+1)$ points in $L$.   Consider the $\binom{r}{k}$ subsets of $U$ of size $k$ and note that $\binom{r}{k} \geq 2$. For each subset $T$ of size $k$, let $\ell_T$ be the least common multiple of the set of points in $T$. If the least common multiple $\ell_T$ is the same for all subsets $T$ of size $k$, then we claim that ${\bf x^w}$ is an exceptional generator of $M_L^{(k)}$. To see this, note that ${\bf x^w} \in M_L^{(k)}$ and any minimal generator $\ell$ of $M_L^{(k)}$ that divides ${\bf x^w}$ dominates every point in some subset of $U$ of size $k$ and $\ell$ is the least common multiple of the Laurent monomials corresponding to points in $U$. However, this least common multiple is ${\bf x^w}$. Hence, $\ell={\bf x^w}$ and is an exceptional generator of $M_L^{(k)}$.

Otherwise, consider two subsets $T_1$ and $T_2$ of $U$ of size $k$ such that their least common multiples $\ell_{T_1}$ and $\ell_{T_2}$ respectively, are different.  There are two cases: 

Either $\ell_{T_1}$ and $\ell_{T_2}$ do not divide each other. Then both $\ell_{T_1}$ and $\ell_{T_2}$ are not equal to ${\bf x^w}$ but divide it. Their supports (the set of lattice points that their exponents dominate) are precisely $T_1$ and $T_2$ respectively (otherwise, this would contradict ${\bf x^w}$ being a minimal generator of $M_L^{(k+1)}$). Hence, $\ell_{T_1}$ and $\ell_{T_2}$ are minimal generators of $M_L^{(k)}$ as any Laurent monomial that divides either $\ell_{T_1}$ or $\ell_{T_2}$ must have strictly smaller support. The map $\phi^{(k)}_S$ takes their syzygy to a monomial $m$ that divides ${\bf x^w}$. Furthermore, since this monomial $m$ is in $M_L^{(k+1)}$ and ${\bf x^w}$ is a minimal generator of $M_L^{(k+1)}$, we conclude that $m={\bf x^w}$.  Finally, note that by Proposition \ref{mlmod_prop2} there is a lattice point ${\bf q} \in L$  such that $\ell_{T_1} \cdot  {\bf x}^{-\bf q}$ and $\ell_{T_2} \cdot {\bf x}^{-\bf q}$ are minimal generators of $M_{L,{\rm mod}}^{(k)}$. Their syzygy maps to an element in the same orbit as ${\bf x^w}$ under the action of $L$.


Suppose that for every pair $\ell_{T_1}$ and $\ell_{T_2}$ one divides the other. Assume that $\ell_{T_1}$ is a proper divisor of $\ell_{T_2}$ and $\ell_{T_1}$ dominates exactly $k$ points in $L$. Then $\ell_{T_2}$ along with the least common multiple of any other subset of size $k$ other than $T_1$ is precisely ${\bf x^w}$ (this is because ${\bf x^w}$ is a minimal generator for $M_L^{(k+1)}$).  Hence, the least common multiple of the set of Laurent monomials with exponents in $T_1 \cup \{{\bf q}\}$ is ${\bf x^w}$ for any ${\bf q} \in T_2 \setminus T_1$.  The map $\phi^{(k)}_S$ takes the syzygy between the minimal generators $\ell_{T_1} \cdot {\bf x^{-q}}$ and $1_{\mathbb{K}}$ of $M_{L,{\rm mod}}^{(k)}$ to an element in the same orbit of ${\bf x^w}$ under the action of the lattice $L$.


\end{proof}

\begin{remark} \label{indchar_rem}
\rm
Note that the proof of Theorem \ref{indchar_theo2} also shows that any element in the image of $\phi_S^{(k)}$ satisfies Case 3 in Theorem \ref{indchar_theo2} i.e. it is also in its image under a syzygy between a minimal generator of $M_L^{(k)}$ and $1_{\mathbb{K}}$. However, those that satisfy Case 2 also carry an $L$-action and hence, we have included this as a separate item in Theorem \ref{indchar_theo2}. \qed
\end{remark}

\begin{example}
\rm
Consider the lattice $L = (3,4,11)^{\perp} \cap \mathbb{Z}^3$. Using our algorithm, we compute its 4th lattice module $M_L^{(4)}$ and as an $S[L]$-module, it equals $\langle  x_3^2, x_1^{-1}x_2x_3^2, x_1^3x_2x_3 \rangle$. The minimal generator $x_1^{-1}x_2x_3^2$ dominates the lattice points $\{ (-1,-2,1), (-2,-4,2), (-6,-1,2), (-5,1,1)\}$. Note that there exists two 3-subsets whose least common multiples are distinct and proper divisors of $x_1^{-1}x_2x_3^2$. We observe that these subsets consist of the first three and last three lattice points, and give the following minimal generators of $M_{L,\text{mod}}^{(3)}$:
\begin{align*}
x_1^{-1}x_2^{-1}x_3^2 &= \text{lcm}(x_1^{-1}x_2^{-2}x_3,x_1^{-2}x_2^{-4}x_3^2,x_1^{-6}x_2^{-1}x_3^2) \\
x_1^{-2}x_2x_3^2 &= \text{lcm}(x_1^{-2}x_2^{-4}x_3^2,x_1^{-6}x_2^{-1}x_3^2,x_1^{-5}x_2x_3)
\end{align*}
Therefore $x_1^{-1}x_2x_3^2$ equals $\phi_S^{(3)}((x_1^{-1}x_2^{-1}x_3^2, x_1^{-2}x_2x_3^2))$ and so is realised as the image of a syzygy between two minimal generators of $M_L^{(3)}$, see Figure \ref{twomingens_fig}.

The minimal generator $x_3^2$ cannot be constructed in this way. It dominates the lattice points $\{ (0,0,0), (-1,-2,1), (-2,-4,2), (-6,-1,2)\}$ where only the least common multiple of the last three lattice points gives a proper divisor of $x_3^2$, specifically $x_1^{-1}x_2^{-1}x_3^2$. This is a minimal generator of $M_{L,{\rm mod}}^{(3)}$ and so $x_3^2$ equals $\phi_S^{(3)}((x_1^{-1}x_2^{-1}x_3^2, 1_{\mathbb{K}}))$, a syzygy between a minimal generator of $M_L^{(3)}$ and $1_{\mathbb{K}}$, as shown in Figure \ref{onemingen_fig}.

For an example of an exceptional generator, we look at the lattice $L = (2,5,10)^{\perp} \cap \mathbb{Z}^3$. The corresponding lattice ideal is $I_L = \left\langle x_3 - x_1^5, x_3 - x_2^2 \right\rangle$, therefore as an $S[L]$-module $M_L^{(2)}$ has generators $x_1^5, x_3, x_2^2$. These all lie in the same $L$-orbit and so $M_L^{(2)}$ is minimally generated by a single element $x_3$. However $x_3$ dominates 3 lattice points $\{(0,0,0), (-5,0,1),(0,-2,1)\}$. Therefore, $x_3$ is an exceptional generator of $M_L^{(2)}$, as shown in Figure \ref{exc_fig}. Indeed, note that the least common multiple of Laurent monomials corresponding to every pair of lattice points is also $x_3$. \qed
\end{example}

\begin{figure}[ht]
  \centering
  \includegraphics[width=6cm]{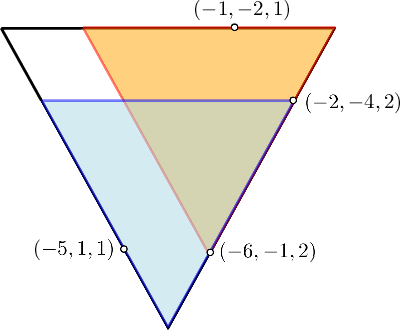}
  \caption{Minimal generator of $M_L^{(4)}$ realised as a syzygy between two minimal generators of $M_L^{(3)}$.}
	\label{twomingens_fig}
\end{figure}
\begin{figure}[ht]
	\centering
  \includegraphics[width=6cm]{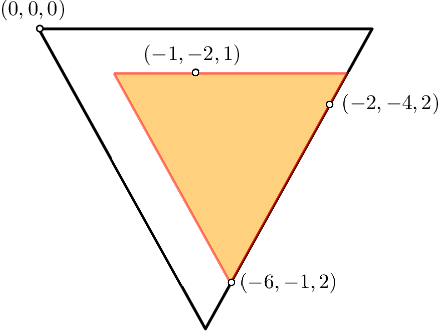}
  \caption{Minimal generator of $M_L^{(4)}$ realised as a syzygy between a minimal generator of $M_L^{(3)}$ and $1_{\mathbb{K}}$.}
	\label{onemingen_fig}
\end{figure}

\begin{figure}[ht]
  \centering
  \includegraphics[width=6cm]{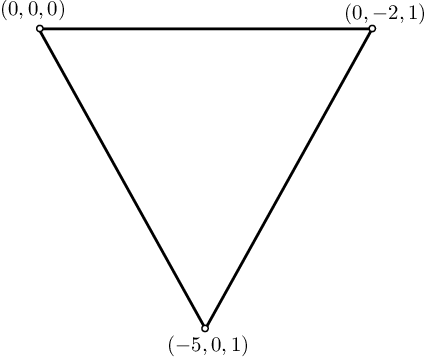}
     \caption{Exceptional generator of $M_{L(2,5,10)}^{(2)}$.}\label{exc_fig}

\end{figure}


Note that $M_{L,{\rm mod}}^{(k)}$ is not finitely generated as an $S$-module and is also not an $S[L]$-module.  This makes Theorem \ref{indchar_theo2} somewhat unwieldy to compute $M_L^{(k)}$. In the following, we use Theorem \ref{indchar_theo2} to prove the neighbourhood theorem that is computationally more amenable.

\subsection{Neighbourhood Theorem}

We briefly recall the graph $G_L$ induced on the lattice $L$. Fix a binomial minimal generating set $B$ of $I_L$. There is an edge between points ${\bf w_1}$ and ${\bf w_2}$ in $L$ if there exists a minimal generator ${\bf x^{u}-x^{v}} \in B$ such that  ${\bf u}-{\bf v}={\bf w_1}-{\bf w_2}$. Let $d_{G_L}$ be the metric on $L$ induced by the graph $G_L$.  For a point ${\bf w} \in L$, we define $N^{(k)}({\bf w})$ to be the set of all points in $L$ in the ball of radius $k$ with respect to the metric $d_{G_L}$ and with ${\bf w}$ as its center.

\begin{theorem}\label{mlkchar_theo}\textup{({\bf Neighbourhood Theorem})}
Any minimal generator of $M_L^{(k)}$ as an $S[L]$-module is the least common multiple of Laurent monomials corresponding to $k$ lattice points in $N^{(k-1)}({\bf 0})$, one of which is the origin. Equivalently, for any minimal generator of $M_L^{(k)}$ as an $S$-module, there is a point ${\bf q} \in L$ such that this minimal generator is the least common multiple of Laurent monomials corresponding to $k$ lattice points in $N^{(k-1)}({\bf q})$,  one of which is ${\bf q}$.
\end{theorem}

In order to prove the theorem, we study certain ``local pieces'' of $G_L$ called the fiber graph.

\begin{definition} \cite[Page 39]{Stu96}
Let $A = (a_1,\dots, a_n)$. For each non-negative integer $b$ we define the set $\mathcal{F}_b = \{ {\bf u} \in \mathbb{Z}_{\geq 0}^n : A \cdot {\bf u} = b$\} to be the \emph{fiber} of $A$ over $b$.
\end{definition}

For any lattice point ${\bf u} \in L$, we can express it uniquely as the difference of positive and negative parts ${\bf u^+} - {\bf u^-}$, where the $i$-th coordinate of ${\bf u^+}$ equals $u_i$ if $u_i > 0$ and equals $0$ otherwise. Since $L$ is contained in $(a_1,\dots,a_n)^{\perp}$, we have ${\bf u^+} \in \mathcal{F}_b$ if and only if ${\bf u^-} \in \mathcal{F}_b$.

We induce a natural graph on the fiber, denoted the fiber graph $G_b$. Fix a binomial minimal generating set $B$ of $I_L$. The vertices of the graph are the elements of the fiber $\mathcal{F}_b$ with an edge between ${\bf w_1}$ and ${\bf w_2}$ if there exists a minimal generator ${\bf x^{u}} - {\bf x^{v}} \in B$  such that ${\bf u}-{\bf v} = {\bf w_1} - {\bf w_2}$. We note that $G_b$ is a finite graph that can be embedded into $G_L$.  The following lemma generalises the statement \cite[Theorem 5.3]{Stu96} that if $I_L$ is a prime ideal (equivalently, if $L$ is a saturated lattice) then $\mathcal{F}_b$ is connected. 

\begin{lemma} \label{concomp}
Let ${\bf u}, {\bf v} \in \mathcal{F}_b$. The difference ${\bf u} - {\bf v}$ is a lattice point in $L$ if and only if ${\bf u}, {\bf v}$ are in the same connected component of $G_b$.
\end{lemma}

\begin{proof}
Suppose ${\bf u} - {\bf v} \in L$, then by definition ${\bf x^u} - {\bf x^v} \in I_L$ and so can be represented as an $S$-linear combination of the minimal generators:
\begin{align} \label{binomialrep}
{\bf x^u} - {\bf x^v} = \sum_{i=1}^N {\bf x^{w_i}} \cdot ({\bf x^{g_i^+}} - {\bf x^{g_i^-}})
\end{align}
We will show by induction on $N$ there exists a path in $G_b$ between ${\bf u}$ and ${\bf v}$. For $N = 1$, expression \eqref{binomialrep} is equivalent to saying that ${\bf u} - {\bf v} = {\bf g_i}$ and so they must be connected by an edge.

Assume the induction hypothesis holds for all $N < N'$, consider expression \eqref{binomialrep} for $N = N'$. We have ${\bf x^u} = {\bf x^{w_i}}\cdot{\bf x^{g_i^+}}$ for some $i$, so without loss of generality we say that ${\bf u} = {\bf w_1 + g_1^+}$, implying ${\bf u}$ and ${\bf w_1 + g_1^-}$ are connected by an edge. Subtracting ${\bf x^{w_1}} \cdot ({\bf x^{g_1^+}} - {\bf x^{g_1^-}})$ from \eqref{binomialrep} gives us an expression of length $N'-1$ for ${\bf x^{w_1 + g_1^-}} - {\bf x^v}$. By the induction hypothesis, these exponents are connected and so ${\bf u}$ and ${\bf v}$ must also be connected.

Conversely, assume that ${\bf u}, {\bf v}$ are in the same connected component of $G_b$. Then there exists some path ${\bf u} = v^{(0)}, v^{(1)}, \dots, v^{(N)} = {\bf v}$ in $G_b$. We can write the binomial
\[
{\bf x^u} - {\bf x^v} = \sum_{i=1}^N {\bf x}^{v^{(i-1)}} - {\bf x}^{v^{(i)}}
\]
where each binomial ${\bf x}^{v^{(i-1)}} - {\bf x}^{v^{(i)}}$ is an element of $I_L$, as $v^{(i-1)}, v^{(i)}$ are connected by an edge. Therefore, ${\bf x^u} - {\bf x^v} \in I_L$ and so ${\bf u} - {\bf v} \in L$.
\end{proof}

\begin{lemma} \label{domfiber}
Let ${\bf v}$ be a lattice point with ${\bf v^+}, {\bf v^-} \in F \subseteq \mathcal{F}_b$, where $F$ is a subset of the fiber $\mathcal{F}_b$ consisting of all elements in the same connected component of $G_b$. The exponent of the least common multiple ${\rm lcm}({\bf x^v}, 1_{\mathbb{K}})$ dominates precisely $|F|$ lattice points, specifically those of the form ${\bf v^+} - {\bf u}$ where ${\bf u} \in F$.
\end{lemma}

\begin{proof}
We first observe that lcm$({\bf x^v}, 1_{\mathbb{K}}) = {\bf x^{v^+}}$. Let ${\bf u} \in F$, then by Lemma \ref{concomp} we deduce that ${\bf v^+} - {\bf u} \in L$. As ${\bf u} \in \mathbb{Z}_{\geq 0}^n$, we see ${\bf v^+} \geq {\bf v^+} - {\bf u}$. This holds for every ${\bf u} \in F$ and so the exponent of lcm$({\bf x^v}, 1_{\mathbb{K}})$ dominates at least $|F|$ lattice points.
Conversely, suppose that for some ${\bf p} \in L$, ${\bf v^+} \geq {\bf p}$. Let ${\bf u}={\bf v^{+}}-{\bf p} \in \mathbb{Z}_{\geq 0}^n$. Then ${\bf v^+} - {\bf u} \in L$, hence by Lemma \ref{concomp} ${\bf u} \in F$.
\end{proof}

\begin{lemma} \label{dompath}
Let ${\bf u}, {\bf v} \in L, d_{G_L}({\bf u}, {\bf v}) = k$. There exists a path of length at least $k$ in $G_L$ from ${\bf u}$ to ${\bf v}$ such that the exponent of {\rm lcm}$({\bf x^u}, {\bf x^v})$ dominates every lattice point on the path.
\end{lemma}

\begin{proof}
As $G_L$ is invariant under translation by $L$, it suffices to prove the case where ${\bf u} = {\bf 0}$. Suppose that ${\bf v^+}, {\bf v^-} \in \mathcal{F}_b$, by Lemma \ref{concomp} they lie in the same connected component of $G_b$ and so there exists a path in $G_b$ given by ${\bf v^+} = v^{(0)}, v^{(1)}, \dots, v^{(n)} = {\bf v^-}$. We can embed this path into $G_L$ by the embedding ${\bf v^+} - v^{(i)}$. This gives us a path from ${\bf 0}$ to ${\bf v}$ in $G_L$ and by Lemma \ref{domfiber} the exponent of lcm$(1_{\mathbb{K}}, {\bf x^v})$ dominates each of the lattice points on this path. As $d_{G_L}({\bf 0},{\bf v}) = k$, this path must be at least length $k$.
\end{proof}

\begin{proof}(Proof of Theorem \ref{mlkchar_theo})
We proceed by induction on $k$. For the base case of $k=1$, the lattice module $M_L^{(1)} = M_L$ has a single generator $1_{\mathbb{K}}$ corresponding to the single lattice point in $N^{(0)}({\bf 0})$. Assume the statement is true for all $k \leq k_0$. Let ${\bf x^u}$ be a minimal generator of $M_L^{(k_0+1)}$, then by Theorem \ref{indchar_theo2} this is either in the image of the map $\phi^{(k_0)}_S$ or is an exceptional generator of $M_L^{(k_0)}$. 

Suppose that it is an exceptional generator of $M_L^{(k_0)}$, then by the inductive hypothesis $\bf{x^u}$ can be expressed as the least common multiple of Laurent monomials corresponding to a set of precisely $k_0$ lattice points, which we denote as $P_{{\bf u}}$. Note that $P_{{\bf u}}$ is a proper subset of the support of $\bf{x^u}$. By lattice translation, we assume that $P_{{\bf u}}$ is contained in $N^{(k_0-1)}({\bf 0})$ and contains ${\bf 0}$. It suffices to show that ${\bf u}$ dominates another lattice point in $N^{(k_0)}({\bf 0})$.

As an exceptional generator ${\bf x^u}$ must dominate at least $k_0+1$ lattice points, so consider a lattice point ${\bf p} \notin P_{{\bf u}}$ that is dominated by ${\bf u}$. If ${\bf p} \in N^{(k_0)}({\bf 0})$, we are done. Suppose ${\bf p} \in N^{(r)}({\bf 0}), r > k_0$. By Lemma \ref{dompath} there exists a path from ${\bf p}$ to ${\bf 0}$ in $G_L$ such that every lattice point in the path is dominated by the exponent of lcm$({\bf x^p}, 1_{\mathbb{K}})$. Therefore there exists some lattice point ${\bf q}$ in this path with $d_{G_L}({\bf 0},{\bf q}) = k_0$ that is dominated by the exponent of lcm$({\bf x^p}, 1_{\mathbb{K}})$. Furthermore as $P_{{\bf u}}$ is contained in $ N^{(k_0-1)}({\bf 0})$, ${\bf q}  \notin P_{{\bf u}}$. As lcm$({\bf x^p}, 1_{\mathbb{K}})$ divides ${\bf x^u}$, it must also dominate all lattice points along this path. Therefore ${\bf x^u}$ can be written as the least common multiple of the Laurent monomials corresponding to the lattice points $P_{{\bf u}} \cup \{{\bf q}\}$ whose cardinality is $k_0 +1$.

Suppose that ${\bf x^u}$ is in the image of $\phi_S^{k_0}$.  According to Remark \ref{indchar_rem}, ${\bf x^u}$ is the image of a syzygy between one minimal generator of $M_L^{(k_0)}$ as an $S$-module and $1_{\mathbb{K}}$. This minimal generator is in the same $L$-orbit as ${\bf x^v}$, a minimal generator of $M_L^{(k_0)}$ satisfying the induction hypothesis. More precisely,  there exists a set $P_{{\bf v}}$ of $k_0$ lattice points whose least common multiple of Laurent monomials equals ${\bf x^v}$ that is contained in $N^{(k_0-1)}({\bf 0})$ and contains ${\bf 0}$. Hence, ${\bf x^u}$ is in the same $L$-orbit as lcm$({\bf x^v}, {\bf x^p})$ for some lattice point ${\bf p}$. It suffices to show that lcm$({\bf x^v}, {\bf x^p})$ satisfies the statement of the theorem.

Let ${\bf p} \in N^{(r)}({\bf 0})$, if $r \leq k_0$ then we are done. Suppose $r > k_0$, by Lemma \ref{dompath} there exists a path from ${\bf 0}$ to ${\bf p}$ in $G_L$ such that every lattice point in the path is dominated by the exponent of $\text{lcm}(1_{\mathbb{K}}, {\bf x^p})$. By the same argument as the previous case, there exists a lattice point ${\bf q}$ on this path with $d_{G_L}({\bf 0},{\bf q}) = k_0$, that is necessarily dominated by ${\bf u}$ and not contained in $P_{{\bf v}}$. Therefore lcm$({\bf x^v}, {\bf x^q})$ is the least common multiple of the Laurent monomials corresponding to $k_0+1$ lattice points $P_{{\bf v}} \cup \{ {\bf q} \}$. The monomial lcm$({\bf x^v}, {\bf x^q})$ divides lcm$({\bf x^v}, {\bf x^p})$, and so is equal to it by the minimality of lcm$({\bf x^v}, {\bf x^p})$. Therefore lcm$({\bf x^v}, {\bf x^p})$ is the least common multiple of $k_0+1$ Laurent monomials corresponding to $P_{{\bf v}} \cup \{ {\bf q} \}$ contained in $N^{(k_0)}({\bf 0})$.

\end{proof}

\section{Finiteness Results}\label{finiteness_sect}

In this section, we show that after suitable twists there are only finitely many isomorphism classes of generalised lattice modules. More precisely, we show the following: 

\begin{theorem}\label{finiteness_theo1}  Let $L$ be a lattice of the form $(a_1,\dots,a_n)^{\perp} \cap \mathbb{Z}^n$. For each $k \in \mathbb{N}$, let ${\bf x^{u_k}}$ be any element $M_L^{(k)}$ of the smallest $(a_1,\dots,a_n)$-weighted degree.  There are finitely many classes among the generalised lattice modules $\{{M_L}^{(k)}{({\bf u_k})}\}_{k \in \mathbb{N}}$ up to isomorphism of both $\mathbb{Z}^n$-graded $S[L]$-modules and $\mathbb{Z}^n$-graded $S$-modules.
\end{theorem}

The main ingredient of the proof of Theorem \ref{finiteness_theo1} is the \emph{structure poset of $L$} that we briefly recall.

{\bf Structure Poset of $L$:} The elements of the structure poset of $L$ are elements in $\mathbb{Z}^n/L$ of $(a_1,\dots,a_n)$-weighted degree in the range $[0,F_1]$ where $F_1$ is the first Frobenius number of $L$. The partial order in this poset is defined as follows: for elements $[{\bf a}],~[{\bf b}]$ in the structure poset we say that $[{\bf a}] \geq [{\bf b}]$ if for every representative ${\bf a} \in \mathbb{Z}^n$ of $[{\bf a}]$ there exists a representative ${\bf b}$ of $[{\bf b}]$ such that ${\bf a} \geq {\bf b}$. Note, $[{\bf a}] \geq [{\bf b}]$ if and only if $[{\bf a}-{\bf b}] \geq [{\bf 0}]$.  Hence, the structure poset of $L$ can be constructed from the set of all elements $[{\bf a}] \geq [{\bf 0}]$ in $\mathbb{Z}^n/L$ whose $(a_1,\dots,a_n)$-weighted degree is in the range $[0,F_1]$. This observation is useful to compute the structure poset.

\begin{example}
\rm
Let $(a_1,a_2,a_3)=(3,5,8)$ and hence, $L(3,5,8)=(3,5,8)^{\perp} \cap \mathbb{Z}^3$. The first Frobenius number is $7$. Hence, the structure poset of $L$ consists of eight elements labelled $0$ to $7$. The poset relations can be determined from the set of all elements that dominate $0$, in this case they are $3,5,6$.  The Haase diagram of the structure poset is shown in Figure \ref{structposet_fig}. \qed

\end{example}

 
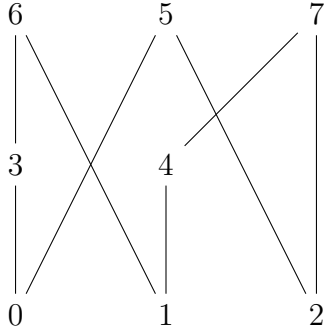
\begin{figure}
\centering
\begin{tikzpicture}
\node (six) at (-2,2) {$6$};
\node (three) at (-2,0) {$3$};
\node (five) at (0,2) {$5$};
\node (four) at (0,0) {$4$};
\node (seven) at (2,2) {$7$};
\node (zero) at (-2,-2) {$0$};
\node (one) at (0,-2) {1};
\node (two) at (2,-2) {$2$};
\draw (six)--(three)--(zero);
\draw (zero)--(five);
\draw (one)--(six);
\draw (one)--(four);
\draw (four)--(seven);
\draw (two)--(seven);
\draw (two)--(five);
\end{tikzpicture}
\caption{The structure poset of $L(3,5,8)$.}
\label{structposet_fig}
\end{figure}

{\bf Structure Poset of $M_L^{(k)}$:} Recall that $m_k$ is the minimum $(a_1,\dots,a_n)$-weighted degree of any element of $M_L^{(k)}$.  A key observation is that $M_L^{(k)}$ is determined (up to isomorphism of $\mathbb{Z}^n$-graded $S[L]$-modules) by the elements in $\mathbb{Z}^n/L$ of weighted degree $[m_k,m_k+F_1]$ that dominate at least $k$ points in $L$. We can see this by considering the submodule ${\bf x^u} \cdot M_L$ of  $M_L^{(k)}$ where ${\bf x^{u}} \in M_L^{(k)}$ has weighted degree $m_k$. Any element of weighted degree greater than $m_k+F_1$ dominates an element with weighted degree $m_k$ in  ${\bf x^u} \cdot M_L$ and so also dominates $k$ lattice points. This observation determines a poset with the same partial order as the structure poset of $L$. Furthermore, by taking $m_k+i$ to $i$,  this determines a subposet of the structure poset of $L$ that we refer to as the \emph{structure poset of $M_L^{(k)}$}. Note that the minimal generators of $M_L^{(k)}$ correspond to the minimal elements of its structure poset.

\begin{proof} (Proof of Theorem \ref{finiteness_theo1})
Note that for any $k$, the $(a_1,\dots,a_n)$-weighted degree of the minimal generators of $M_L^{(k)}$ are in the range $[m_k,m_k+F_1]$. Furthermore, the structure poset of $M_L^{(k)}$ as  a subposet of the structure poset of $L$ determines $M_L^{(k)}({\bf u_k})$ up to isomorphism of $\mathbb{Z}^n$-graded $S[L]$-modules (and $\mathbb{Z}^n$-graded $S$-modules).  More precisely, if $M_L^{(k_1)}$ and $M_L^{(k_2)}$ have the same structure poset, then  multiplying $M_L^{(k_1)}({\bf u_{k_1}})$ by the Laurent monomial ${\bf x^{u_2}}/{\bf x^{u_1}}$ is an isomorphism between $M_L^{(k_1)}({\bf u_{k_1}})$ and $M_L^{(k_2)}({\bf u_{k_2}})$ (as both $\mathbb{Z}^n$-graded $S[L]$-modules and $\mathbb{Z}^n$-graded $S$-modules). In particular, this map induces a bijection between the (monomial) minimal generating set of $M_L^{(k_1)}({\bf u_{k_1}})$ and the (monomial) minimal generating set of $M_L^{(k_2)}({\bf u_{k_2}})$ and preserves degrees.   Since the structure poset of $L$ is finite, it has only finitely many subposets. Hence, there are only finitely many $\mathbb{Z}^n$-graded isomorphism classes of the twisted generalised lattice modules $\{M_L^{(k)}({\bf u_k})\}_{k=1}^{\infty}$.
\end{proof}






Theorem \ref{finiteness_theo1} and its proof also generalises to finite index sublattices $L$ of $(a_1,\dots,a_n)^{\perp} \cap \mathbb{Z}^n$. The only additional subtlety is that the structure poset of $M_L^{(k)}$ will have precisely as many embeddings into the structure poset of $L$ as the number of elements of weighted degree $m_k$ in $M_L^{(k)}$. If $M_L^{(k_1)}$ and $M_L^{(k_2)}$ have the same embedding into the structure poset of $L$, then we have exactly the same isomorphism as in the proof of Theorem \ref{finiteness_theo1}. There are still only finitely many subposets of the structure poset of $L$. 

\begin{remark}
It is worth noting that not all subposets of the structure poset of $L$ can be realised as the structure poset of some $M_L^{(k)}$. If an element is contained in the structure poset of some $M_L^{(k)}$, all elements greater than it according to the partial order must also be contained in it. Therefore the poset is completely determined by its set of minimal elements, which form an antichain of the structure poset of $L$. As a result, the number of subposets realisable as the structure poset of some $M_L^{(k)}$ is upper bounded by the number of antichains of the structure poset of $L$. For counting the number of antichains, tools such as Dilworth's theorem \cite{Dil50} are useful.
\end{remark}

\begin{remark}
The data of the structure poset of $M_L^{(k)}$ where $L = L(a_1,\dots,a_n)$ is encoded in the Hilbert series of the polynomial ring $S$ with the $(a_1,\dots,a_n)$-weighted grading. The elements of $M_L^{(k)}$ are those $j$ such that the Hilbert coefficient $h_j$ is at least $k$. This Hilbert series is also referred to as the restricted partition function in \cite[Page 6]{BecRob09} and is a useful tool for explicitly computing the structure poset.  Note that for a finite index sublattice $L$ of $(a_1,\dots,a_n)^{\perp} \cap \mathbb{Z}^n$, this data is encoded in the Hilbert series of $S$ with the $\mathbb{Z}^n/L$-grading.
\end{remark}

\begin{example}
In the following, we compute the structure poset of $M_{L(3,5,8)}^{(k)}$ for $k$ from $1$ to $6$. The Hilbert series of the polynomial ring with the $(a_1,\dots,a_n)$-weighted grading is given by the rational function $\frac{1}{(1-t)^{a_1}\cdots (1-t)^{a_n}}$. Using this information, we determine $m_1,\dots,m_6$ to be $0,8,16,21,24,29$. The other elements of the structure poset of $M_{L(3,5,8)}^{(k)}$ are the integers $i$ in the interval $[0, 7]$ such that $h_{m_k +i} \geq k$. The corresponding structure posets are shown in Figure \ref{structposetmlk_fig}.

 \qed
\end{example}

\begin{figure}[ht]
\begin{subfigure}{0.49\textwidth}
\label{structposetml1_fig}
\centering
\begin{tikzpicture}
	\node (six) at (-2,2) {$6$};
	\node (three) at (-2,0) {$3$};
	\node (five) at (0,2) {$5$};
	\node (zero) at (-1,-2) {$0$};
  \draw (six)--(three)--(zero);
	\draw (zero)--(five);
\end{tikzpicture}
\caption{$k=1$}
\end{subfigure}
\begin{subfigure}{0.49\textwidth}
\label{structposetml2_fig}
\centering
\begin{tikzpicture}
	\node (six) at (-2,2) {$6$};
	\node (three) at (-2,0) {$3$};
	\node (five) at (0,2) {$5$};
	\node (zero) at (-1,-2) {$0$};
	\node (seven) at (2,2) {$7$};
	\draw (six)--(three)--(zero);
	\draw (zero)--(five);
\end{tikzpicture}
\caption{$k=2$}
\end{subfigure}
\begin{subfigure}{0.49\textwidth}
\label{structposetml3_fig}
\centering
\begin{tikzpicture}
	\node (six) at (-2,2) {$6$};
	\node (three) at (-2,0) {$3$};
	\node (five) at (0,2) {$5$};
	\node (four) at (0,0) {$4$};
	\node (seven) at (2,2) {$7$};
	\node (zero) at (-1,-2) {$0$};
	\node (two) at (1,-2) {$2$};
	\draw (six)--(three)--(zero);
	\draw (zero)--(five);
	\draw (four)--(seven);
	\draw (two)--(seven);
	\draw (two)--(five);
\end{tikzpicture}
\caption{$k=3$}
\end{subfigure}
\begin{subfigure}{0.49\textwidth}
\label{structposetml4_fig}
\centering
\begin{tikzpicture}
	\node (six) at (-2,2) {$6$};
	\node (three) at (-2,0) {$3$};
	\node (five) at (0,2) {$5$};
	\node (four) at (0,0) {$4$};
	\node (seven) at (2,2) {$7$};
	\node (zero) at (-1,-2) {$0$};
	\node (two) at (1,-2) {$2$};
	\draw (six)--(three)--(zero);
	\draw (zero)--(five);
	\draw (four)--(seven);
	\draw (two)--(seven);
	\draw (two)--(five);
\end{tikzpicture}
\caption{$k=4$}
\end{subfigure}
\begin{subfigure}{0.49\textwidth}
\label{structposetml5_fig}
\centering
\begin{tikzpicture}
	\node (six) at (-2,2) {$6$};
	\node (three) at (-2,0) {$3$};
	\node (five) at (0,2) {$5$};
	\node (four) at (0,0) {$4$};
	\node (seven) at (2,2) {$7$};
	\node (zero) at (-1,-2) {$0$};
	\node (two) at (1,-2) {$2$};
	\draw (six)--(three)--(zero);
	\draw (zero)--(five);
	\draw (four)--(seven);
	\draw (two)--(seven);
	\draw (two)--(five);
\end{tikzpicture}
\caption{$k=5$}
\end{subfigure}
\begin{subfigure}{0.49\textwidth}
\label{structposetml6_fig}
\centering
\begin{tikzpicture}
	\node (six) at (-2,2) {$6$};
	\node (three) at (-2,0) {$3$};
	\node (five) at (0,2) {$5$};
	\node (four) at (0,0) {$4$};
	\node (seven) at (2,2) {$7$};
	\node (zero) at (-2,-2) {$0$};
	\node (two) at (2,-2) {$2$};
	\node (one) at (0,-2) {$1$};
	\draw (six)--(three)--(zero);
	\draw (zero)--(five);
	\draw (one)--(six);
	\draw (one)--(four);
	\draw (four)--(seven);
	\draw (two)--(seven);
	\draw (two)--(five);
\end{tikzpicture}
\caption{$k=6$}
\end{subfigure}
\caption{The structure posets of $M_{L(3,5,8)}^{(k)}$ for $k$ from $1$ to $6$.}
\label{structposetmlk_fig}
\end{figure}
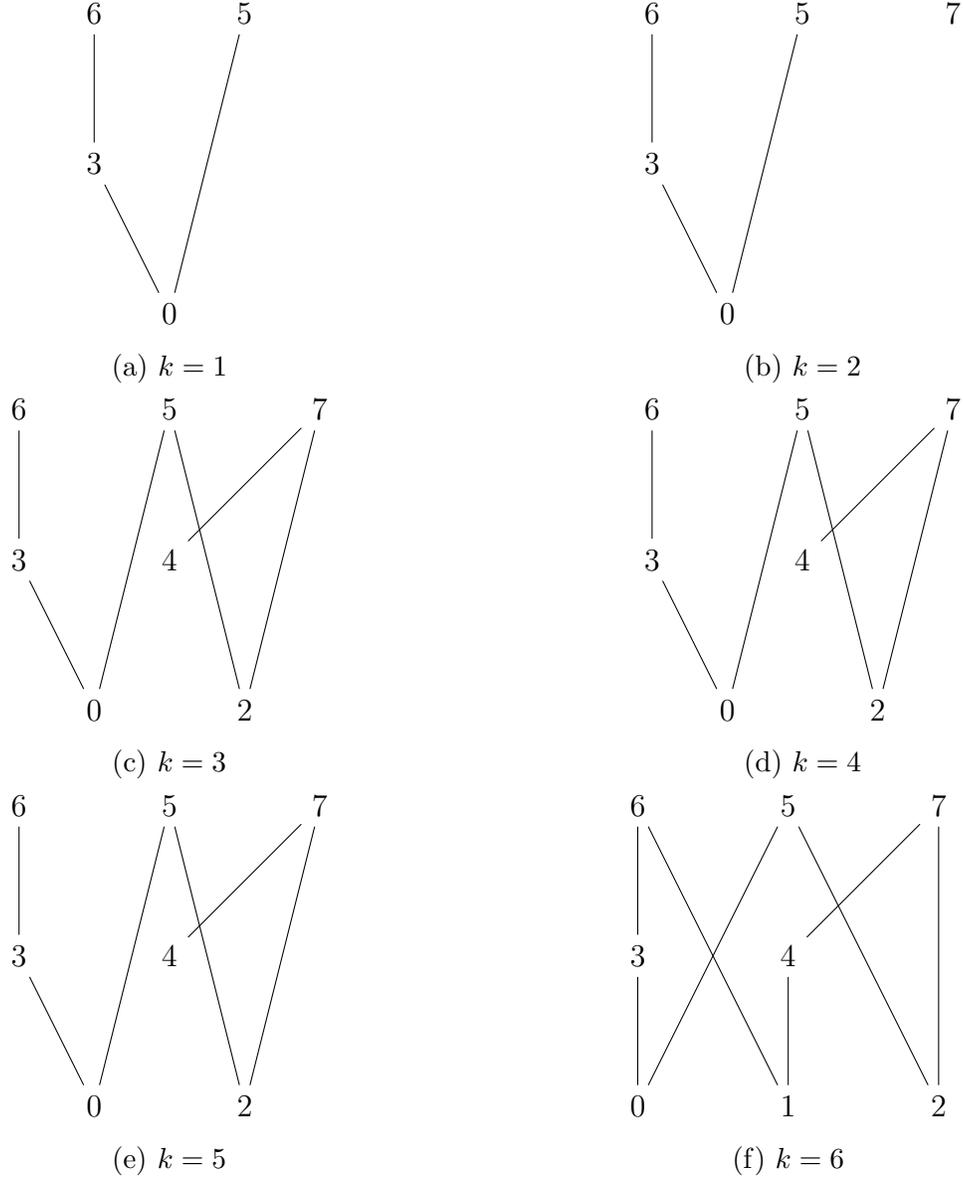


Based on the same ideas as in Theorem \ref{finiteness_theo1}, we obtain the following upper bounds on generalised Frobenius numbers and the number of minimal generators of $M_L^{(k)}$.

\begin{proposition}
The $k$-th Frobenius number $F_k$ is upper bounded by $m_k+F_1$. The number $\beta_1(M_L^{(k)})$ of minimal generators of $M_L^{(k)}$ as an $S[L]$-module is upper bounded by the maximum length of an antichain in the structure poset of $L$.   
\end{proposition}

Furthermore, we have the following corollary to Theorem \ref{finiteness_theo1}.



\begin{corollary}\label{finitenessfrob_corsect}
There exists a finite set of integers $\{b_1,\dots,b_t\} \subset \mathbb{Z}_{\geq 0} \cup \{-1\}$  such that for every $k$ there exists a natural number $j$ such that the $k$-th Frobenius number can be written as:  \begin{center} $F_k=m_k+b_j$ \end{center} where $m_k$ is the minimum  $(a_1,\dots,a_n)$-weighted degree of an element in $M_L^{(k)}$. This finite set $\{b_1,\dots,b_t\}$ is the precisely the set of  integers that can be realised as ${\rm reg}(M_L^{(k)}({\bf u_k}))+n-1-\sum_{i=1}^n a_i$.
\end{corollary}


\section{Applications}

\subsection{The Sequence of Generalised Frobenius Numbers}

We prove that the sequence of generalised Frobenius numbers form a finite difference progression. 

\begin{definition}
A sequence $(c_k)_{k=1}^{\infty}$ is called a \emph{finite difference progression} if there exists a finite set of differences such that for every $k \in \mathbb{N}$ the difference $c_{k+1}-c_{k}$ is contained in this set. The \emph{rank} of the progression is defined to be the cardinality of this set.
\end{definition}

\begin{theorem}\label{genarith_theo}
For any finite index sublattice $L$ of $(a_1,\dots,a_n)^{\perp} \cap \mathbb{Z}^n$, the sequence of generalised Frobenius numbers $(F_k)_{k=1}^{\infty}$ is a finite difference progression.
 \end{theorem} 
We note that this follows immediately from Corollary \ref{finitenessfrob_corsect} once we show that the sequence $(m_k)_{k=1}^{\infty}$ is also a finite difference progression.

\begin{lemma}\label{mkgenarith_lem}
For any finite index sublattice $L$ of $(a_1,\dots,a_n)^{\perp} \cap \mathbb{Z}^n$, the sequence $(m_k)_{k=1}^{\infty}$ is a finite difference progression. 
\end{lemma}
\begin{proof}
We show the difference of successive terms is bounded by $0 \leq m_{k+1}-m_k \leq m_2$, and therefore the set of successive differences is finite. The inequality $m_{k+1}-m_k \geq 0$ follows by construction. 

To prove the other bound,  we construct an element of degree at most $m_k+m_2$ in $M_L^{(k+1)}$ and hence, conclude that $m_{k+1}-m_k \leq m_2$.  Consider a minimal generator of $M_L^{(2)}$ of weighted degree $m_2$ that dominates the origin and another lattice point ${\bf p}$. Note that this minimal generator is ${\bf x^{p^+}}$. 

Consider a minimal generator ${\bf x^q}$ of $M_L^{(k)}$  of weighted degree $m_k$, such that the origin is in its support and ${\bf p}$ is not in its support. Note that such a generator exists by the following lattice translation argument. Take any minimal generator  ${\bf x^{q'}}$ of $M_L^{(k)}$ of weighted degree $m_k$ and maximise the linear functional ${\bf p} \cdot {\bf x}$ over its support. Suppose that ${\bf r}$ is a point in the support at which this functional is maximised, multiply the minimal generator by ${\bf x^{-r}}$. The resulting minimal generator contains the origin but does not contain the point ${\bf p}$ in its support. This is because the origin maximises the functional  ${\bf p} \cdot {\bf x}$  over the support of ${\bf x^{q'}} \cdot {\bf x^{-r}}$ and the inner product of ${\bf p}$ with the origin is zero whereas its inner product with itself is strictly positive. 

The monomial ${\rm lcm}({\bf x^{p^{+}}},{\bf x^q})$ is contained in $M_L^{(k+1)}$ as its support contains the union of supports of ${\bf x^q}$ and ${\bf x^{p}}$, and has weighted degree at most $m_2+m_{k}$. As an element of $M_L^{(k+1)}$ it must have weighted degree at least $m_{k+1}$ and therefore $m_{k+1}-m_k \leq m_2$.
\end{proof}

The sequence $(m_k)_{k=1}^{\infty}$ inherits much of the structure of $M_L^{(k)}$ given by its inductive characterisation (Theorem \ref{indchar_theo2}). This additional structure makes it more natural to derive bounds on successive differences rather than $(F_k)_{k=1}^{\infty}$ directly.

Recall that the rank of the finite difference progression is defined as the cardinality of its set of successive differences. Note that the rank is equal to one when the sequence is an arithmetic progression. Given the sequence of $k$-th Frobenius numbers $(F_k)_{k=1}^{\infty}$ with associated $\{b_1,\dots,b_t\}$ such that $b_t \geq b_{t-1} \geq \dots \geq b_1$ (as defined in Corollary \ref{finitenessfrob_corsect}), we derive two upper bounds on its rank from Lemma \ref{mkgenarith_lem} and Corollary \ref{finitenessfrob_corsect}.

\begin{proposition}
The rank of the finite difference progression  $(F_k)_{k=1}^{\infty}$ is upper bounded by:
\begin{align}
\textup{rank}((F_k)_{k=1}^{\infty}) &\leq m_2 + b_t - b_1 + 1 \label{bound1_eqn} \\
\textup{rank}((F_k)_{k=1}^{\infty}) &\leq \left({t \choose 2} + 1\right)(m_2 + 1) \label{bound2_eqn}
\end{align}
\end{proposition}
\begin{proof}
Bound \eqref{bound1_eqn} is derived from the fact that the largest possible difference between successive terms is $m_2 + b_t - b_1$. This is possible when $F_k = m_k + b_1$ and $F_{k+1} = m_{k+1} + b_t$ where $m_{k+1} - m_k = m_2$, the largest possible difference as shown in the proof of Lemma \ref{mkgenarith_lem}. All possible differences are in the interval $[0, m_2 + b_t - b_1]$ and so the rank is upper bounded by its cardinality.

Bound \eqref{bound2_eqn} is derived as follows. By Corollary \ref{finitenessfrob_corsect}, we can express the difference $F_{k+1} - F_{k} = (m_{k+1} - m_k) + (b_j - b_i)$ for some $b_i, b_j \in \{b_1,\dots,b_t\}$. Recall from Lemma \ref{mkgenarith_lem} that the set of differences $\{m_{k+1} - m_k\}_{k\in \mathbb{N}}$ is a subset of $[0,m_2]$. We consider the following two cases:
\begin{enumerate}[label=Case \arabic*]
\item {\bf ($b_j > b_i$):} There are ${t \choose 2}$ choices of $b_i, b_j$ that satisfy $b_j > b_i$ and so the number of differences $\{b_j - b_i\}_{i <j}$ is upper bounded by ${t \choose 2}$. Therefore the number of differences $\{F_{k+1} - F_k\}$ is upper bounded by ${t \choose 2}(m_2 + 1)$.

\item {\bf ($b_j \leq b_i$):} Here $0 \leq F_{k+1} - F_k \leq m_{k+1} - m_k$, therefore the set of differences is a subset of $[0,m_2]$.
\end{enumerate}
Summing up the upper bounds over both cases, we get the bound $\text{rank}((F_k)_{k=1}^{\infty}) \leq ({t \choose 2} + 1)(m_2 + 1)$
\end{proof}

\begin{corollary}\label{genfrobgeom_cor} A geometric progression with common ratio strictly greater than one cannot occur as a sequence of generalised  Frobenius numbers of any finite index sublattice of $(a_1,\dots,a_n)^{\perp} \cap \mathbb{Z}^n$.\end{corollary}
\begin{proof}  By Theorem \ref{genarith_theo}, a sequence of generalised Frobenius numbers  $(F_k)_{k=1}^{\infty}$  is a finite difference progression. Hence, the difference  $F_{k+1}-F_k$  is uniformly upper bounded. On the other hand, since the common ratio of the geometric progression is greater than one, the difference between successive terms goes to infinity with $k$. Hence, such a geometric progression cannot occur as a sequence of generalised Frobenius numbers.
\end{proof}

\begin{remark}
\rm 
Another reason to expect Corollary \ref{genfrobgeom_cor} is that the sequence of generalised Frobenius numbers of lattices of dimension at least two usually contains plenty of repetitions. However, Theorem \ref{genarith_theo} implies a stronger statement that even after removing the repetitions the resulting sequence cannot be a geometric progression of common ratio strictly greater than one. \qed
\end{remark}

\subsection{Algorithms for Generalised Frobenius Numbers}

We use the Neighbourhood theorem (Theorem \ref{mlkchar_theo}) to give an algorithmic construction of generalised lattice modules and via Proposition \ref{kfroblattmod_prop} compute generalised Frobenius numbers.

\begin{algorithm}
\caption{Generalised Lattice Modules}

{\bf Input:} A  basis of a finite index sublattice $L$ of $(a_1,\dots,a_n)^{\perp} \cap \mathbb{Z}^n$ where $(a_1,\dots,a_n) \in \mathbb{N}^n$ and a natural number $k \in \mathbb{N}$.

{\bf Output:} A minimal generating set of $M_L^{(k)}$ as an $S[L]$-module and the $k$-th Frobenius number $F_k$ of $L$.

\begin{algorithmic}[1]

\\ Compute the lattice ideal $I_L$.

\\ Compute all lattice points in $N^{(k-1)}({\bf 0})$.

\\ For each $k$-subset $P \subseteq N^{(k-1)}({\bf 0})$ containing ${\bf 0}$, calculate the least common multiple $\ell_P = \text{lcm}({\bf x^{p_i}} \;| \; {\bf p_i} \in P)$.

\\ Construct $M_L^{(k)} = \left\langle \ell_P \;| \; P \subseteq N^{(k-1)}({\bf 0})\,, ~|P|=k,~{\bf 0} \in P \right\rangle_{S[L]}$.

\\ Pick a representative that is minimal under divisibility for each $L$-orbit and declare the resulting set to be a minimal generating set of $M_L^{(k)}$.

\\ Compute the $\mathbb{Z}^n/L$-graded $S$-module $\pi(M_L^{(k)}):=M_L^{(k)} \otimes_{S[L]}S$ and its Castelnuovo--Mumford regularity ${\rm reg}(\pi(M_L^{(k)}))$.

\\ Set the $k$-th Frobenius number $F_k$ to $\text{reg}(\pi(M_L^{(k)}))+n-1-\sum_{i=1}^{n} a_i$.
\end{algorithmic}
\end{algorithm}



\begin{remark}
\rm
A method for computing the lattice ideal given a basis for that lattice is presented in \cite{MilStu05}. One method to compute the Castelnuovo--Mumford regularity of $\pi(M_L^{(k)})$ is to construct a free presentation of $\pi(M_L^{(k)})$, for instance via the hull complex of $M_L^{(k)}$. We can use this as the input to the algorithm presented in \cite{BaySti92} to compute the Castelnuovo--Mumford regularity.  \qed
\end{remark}



\begin{example} \rm
\label{ex:(3,4,11)}

In the following example, we illustrate our algorithm in the case where the lattice $L=L(3,4,11)=(3,4,11)^{\perp} \cap \mathbb{Z}^3$ and $k=3$. Figure \ref{fig:MonomialStaircase} shows the monomial staircase for this lattice module.

The set $\{(1,2,-1), (4,-3,0)\}$ is a basis for $L$. The binomials corresponding to this basis generate the ideal $J = \left\langle x_1x_2^2 - x_3, x_1^4 - x_2^3 \right\rangle$. The lattice ideal $I_L$ is given by the saturation of $J$ with respect to the product of all the variables, and so
\[
I_L = \left\langle J : \left\langle x_1x_2x_3 \right\rangle^{\infty} \right\rangle = \left\langle x_1x_2^2 - x_3, x_1^4 - x_2^3 \right\rangle.
\]
In this case, the lattice ideal does not have any new binomials.

The lattice points $(1,2,-1), (4,-3,0)$ along with their negative and the origin ${\bf 0}$, give the first neighbourhood $N^{(1)}({\bf 0})$. Next, we compute $N^{(k-1)}({\bf 0})$ by taking all $k$-subsets of $N^{(1)}({\bf 0})$ and taking their sum. This computation gives us
\begin{align*}
N^{(2)}({\bf 0}) = &\{(0, 0, 0), (1, 2, -1), (4, -3, 0), (-1, -2, 1), (-4, 3, 0), (8, -6, 0), (3, -5, 1), \\
&(5, -1, -1), (-2, -4, 2), (2, 4, -2), (-5, 1, 1), (-3, 5, -1), (-8, 6, 0)\}.
\end{align*}
For each $3$-subset of $N^{(2)}({\bf 0})$, we take the least common multiple of the corresponding monomials and denote the $S[L]$-module generated by these monomials as $M_{\text{con}}$. By the Neighbourhood theorem, $M_{\text{con}}$ is equal to $M_L^{(3)}$. Note that this requires computing ${12 \choose 2} = 66$ monomials.

To calculate a minimal generating set of $M_L^{(3)}$, we choose the monomials from this set that do not dominate any other monomial in $M_L^{(3)}$. In our case, this gives the following list of generators
\[
M_L^{(3)} = \left\langle x_1^5, x_1^4x_2^2, x_1x_2^3, x_1^3x_3, x_2^5, x_2x_3\right\rangle_{S[L]}.
\]
All minimal generators with the same $\mathbb{Z}^n/L$-degree must be in the same $L$-orbit. Hence, we pick representatives for each degree to give a minimal generating set of $M_L^{(3)}$. All minimal generators are in degree 15 or 20, and so $M_L^{(3)} = \left\langle x_1^5, x_1^4x_2^2 \right\rangle_{S[L]}$. We compute the Castelnuovo--Mumford regularity of $\pi(M_L^{(3)}) = 33$. Therefore, we calculate $F_3$ to be
\[
 33 + 2 -3 -4 -11 = 17
\]
\qed
\end{example}

\section{Future Directions}

We organise potential future directions into three items with the first two closely related.

\begin{itemize}

\item {\bf Classification of Sequences of Generalised Frobenius Numbers:} We have shown that the sequence of generalised Frobenius numbers form a finite difference progression, however there is still information that we have not fully utilised. For instance,  we have not used the filtration of the generalised lattice modules and the inductive characterisation provided by Theorem \ref{mlkchar_theo}.  Can this information be used to study sequences of generalised Frobenius numbers? For instance, by studying the sequence of Castelnuovo--Mumford regularity of modules in a filtration. 

\item {\bf Syzygies of Generalised Lattice Modules:} Our finiteness result shows that for any finite index sublattice of $(a_1,\dots,a_n)^{\perp} \cap \mathbb{Z}^n$ there are only finitely many isomorphism classes of generalised lattice modules. What are the possible Betti tables that can occur as Betti tables of generalised lattice modules? How are they related? Note that this is closely related to the previous item since the Castelnuovo--Mumford regularity of $M_L^{(k)}$ is the number of rows of its Betti table minus one and this is essentially the $k$-th Frobenius number (Proposition \ref{kfroblattmod_prop}). This problem is also closely related to the problem of classifying structure posets of generalised lattice modules (see Section \ref{finiteness_sect} for more details).

Peeva and Sturmfels  \cite{PeeStu98} define a notion of lattice ideals associated to generic lattices and show that the Scarf complex minimally resolves lattice ideals associated to generic lattices. For any fixed $k$ and a generic lattice $L$, is there a generalisation of the Scarf complex to a complex that minimally resolves $M_L^{(k)}$ as an $S[L]$-module?

\item {\bf Generalised Frobenius Numbers of Laplacian Lattices:}  Let $G$ be a labelled graph. Recall that the Laplacian matrix $Q(G)$ is the matrix $D-A$ where $D$ is the diagonal matrix  ${\rm diag}({\rm val}(v_1),\dots,{\rm val}(v_n))$ where ${\rm val}(v_i)$ is the valency of the vertex $v_i$ and $A$ is the vertex-vertex adjacency matrix.  The Laplacian lattice $L_G$ of $G$ is the lattice generated by the rows of the Laplacian matrix.  This is a finite index sublattice of the root lattice $A_{n-1}=(1,\dots,1)^{\perp} \cap \mathbb{Z}^n$ of index equal to the number of spanning trees of $G$. We know from \cite{AmiMan10} that the first Frobenius number of $L_G$ is equal to the genus of the graph. The genus of the graph is its first Betti number as a simplicial complex of dimension one and is equal to $m-n+1$ where $m$ is the number of edges.  Is this there a generalisation of this interpretation to generalised Frobenius numbers? 

 Arithmetical graphs are generalisations of graphs motivated by applications from arithmetic geometry, see Lorenzini \cite{Lor12}. Lorenzini associated a Laplacian lattice to an arithmetical graph and defines its genus as the first Frobenius number of its Laplacian lattice. He studies it in the context of the Riemann--Roch theorem. The generalised Frobenius numbers of Laplacian lattices associated to arithmetical graphs seems another fruitful future direction.

\end{itemize}

\noindent
\emph{Acknowledgements.} We thank Jes\'us De Loera for interesting discussions and his encouragement during the Summer School on convex geometry organised by the Berlin Mathematical School in Berlin in June 2015. 
This project was inspired by these discussions. We also thank the organisers Martin Henk and Raman Sanyal. Additional thanks to Bernd Sturmfels and Spencer Backman for stimulating conversations and encouragement.  We thank the anonymous referee for  several constructive suggestions and Vic Reiner for his encouragement.  We acknowledge the computer algebra system Macaulay2 \cite{M2} for both investigation and preparation of examples.

\nocite{*}
\bibliographystyle{plain}
\bibliography{genfrobcommalg}
\end{document}